\documentclass[reqno]{amsproc}
\usepackage{amsthm}
\usepackage{amsfonts}
\usepackage{amsmath}
\usepackage{amssymb}
\usepackage{mathrsfs}
\usepackage{epsfig}
\usepackage{tabularx}
\usepackage{makecell}
\usepackage{graphicx}
\usepackage{float}
\usepackage{epstopdf}
\usepackage{subfigure}
\usepackage{color}
\usepackage{mathbbol}
\usepackage{pifont}

\makeatletter
\@namedef{subjclassname@2010}{%
\textup{2010} Mathematics Subject Classification}
\makeatother

\newtheorem{theorem}{Theorem}[section]
\newtheorem{proposition}[theorem]{Proposition}
\newtheorem{corollary}[theorem]{Corollary}
\newtheorem{lemma}[theorem]{Lemma}

\newtheorem*{opq*}{\bf Problem}

\theoremstyle{remark}
\newtheorem{remark}[theorem]{Remark}

\theoremstyle{definition}
\newtheorem{example}[theorem]{Example}

\newcommand{\B}{\boldsymbol B}

\newcommand{\bq}{\begin{equation}}
\newcommand{\eq}{\end{equation}}
\newcommand{\beqn}{\begin{eqnarray*}}
\newcommand{\eeqn}{\end{eqnarray*}}
\newcommand{\beq}{\begin{eqnarray}}
\newcommand{\eeq}{\end{eqnarray}}

\newcommand{\bc}{\begin{centre}}
\newcommand{\ec}{\end{centre}}

\newcommand{\ba}{\begin{array}}
\newcommand{\ea}{\end{array}}

\newcommand{\inp}[2]{\langle{#1},\,{#2} \rangle}

\def \C{\mathbb{C}}

\newcommand*{\borel}[1]{{\mathfrak B}(#1)}
\newcommand*{\child}[1]{\mathsf{Chi}(#1)}
\newcommand*{\childs}[2]{\mathsf{Chi}_{#1}(#2)}
\newcommand*{\childn}[2]{{\mathsf{Chi}}^{\langle#1\rangle}(#2)}
\newcommand*{\childns}[3]{{\mathsf{Chi}}_{#1}^{\langle#2\rangle}(#3)}
\newcommand*{\des}[1]{{{\mathsf{Des}}(#1)}}
\newcommand*{\dess}[2]{{{\mathsf{Des}}_{#1}(#2)}}
\newcommand*{\gammab}{\boldsymbol\gamma}
\newcommand*{\Ge}{\geqslant}
\newcommand*{\hh}{\mathcal{H}}
\newcommand*{\kk}{\mathcal{K}}
\newcommand*{\lambdab}{\boldsymbol\lambda}
\newcommand*{\Le}{\leqslant}
\newcommand*{\nbb}{\mathbb N}
\newcommand*{\parent}[1]{\mathsf{par}(#1)}
\newcommand*{\parents}[2]{\mathsf{par}_{#1}(#2)}
\newcommand*{\ran}{\mathrm{ran\,}}
\newcommand*{\rbb}{\mathbb R}
\newcommand*{\rot}{\mathsf{\omega}}
\newcommand*{\slam}{S_{\lambdab}}
\newcommand*{\slamj}{S_{\mathbb 1}}

\newcommand*{\supp}[1]{\mathrm{supp}(#1)}
\newcommand*{\tcal}{\mathscr T}
\newcommand*{\zbb}{\mathbb Z}
   
\begin{document}
   \title[A solution to the Cauchy dual subnormality problem]
{A solution to the Cauchy dual subnormality \\ problem
for $2$-isometries}
   \author[A. Anand]{Akash Anand}
   \address{Department of Mathematics and Statistics\\
Indian Institute of Technology Kanpur, India}
   \email{akasha@iitk.ac.in}
   \author[S. Chavan]{Sameer Chavan}
   \address{Department of Mathematics and Statistics\\
Indian Institute of Technology Kanpur, India}
   \email{chavan@iitk.ac.in}
   \author[Z.\ J.\ Jab{\l}o\'nski]{Zenon Jan
Jab{\l}o\'nski}
   \address{Instytut Matematyki,
Uniwersytet Jagiello\'nski, ul.\ \L ojasiewicza 6,
PL-30348 Kra\-k\'ow, Poland}
\email{Zenon.Jablonski@im.uj.edu.pl}
   \author[J.\ Stochel]{Jan Stochel}
\address{Instytut Matematyki, Uniwersytet
Jagiello\'nski, ul.\ \L ojasiewicza 6, PL-30348
Kra\-k\'ow, Poland} \email{Jan.Stochel@im.uj.edu.pl}
   \thanks{The research of the
third and fourth authors was supported by the NCN
(National Science Center), decision No.
DEC-2013/11/B/ST1/03613.}
   \begin{abstract}
The Cauchy dual subnormality problem asks whether the
Cauchy dual operator $T^{\prime}:=T(T^*T)^{-1}$ of a
$2$-isometry $T$ is subnormal. In the present paper we
show that the problem has a negative solution. The
first counterexample depends heavily on a
reconstruction theorem stating that if $T$ is a
$2$-isometric weighted shift on a rooted directed tree
with nonzero weights that satisfies the perturbed
kernel condition, then $T^{\prime}$ is subnormal if
and only if $T$ satisfies the (unperturbed) kernel
condition. The second counterexample arises from a
$2$-isometric adjacency operator of a locally finite
rooted directed tree again by thorough investigations
of positive solutions of the Cauchy dual subnormality
problem in this context. We prove that if $T$ is a
$2$-isometry satisfying the kernel condition or a
quasi-Brownian isometry, then $T^{\prime}$ is
subnormal. We construct a $2$-isometric adjacency
operator $T$ of a rooted directed tree such that $T$
does not satisfy the kernel condition, $T$ is not a
quasi-Brownian isometry and $T^{\prime}$ is subnormal.
   \end{abstract}

   \maketitle

   \section{Introduction}
Let $\hh$ be a (complex) Hilbert space. Denote by
$\B(\hh)$ the $C^*$-algebra of all bonded linear
operators on $\hh$. An operator $S \in \B(\hh)$ is
said to be {\it subnormal} if there exist a Hilbert
space $\kk$ containing $\hh$ and a normal operator
$N\in \B(\kk)$ such that $Nh = Sh$ for every $h \in
\hh.$ An operator $T \in \B(\hh)$ is called {\it
hyponormal} if the self-commutator $T^*T-TT^*$ of $T$
is positive. We recall that every subnormal operator
is hyponormal, but not conversely. For a comprehensive
account on the theory of subnormal and hyponormal
operators, the reader is referred to~ \cite{Co}.

Given a positive integer $m$, we say that an operator
$T\in \B(\hh)$ is an {\it $m$-isometry} (or {\em
$m$-isometric}) if $B_m(T) = 0$, where
   \begin{align*}
B_m(T) = \sum_{k=0}^m (-1)^k {m \choose k}{T^*}^kT^k.
   \end{align*}
Clearly, a $1$-isometry is an isometry. We say that
$T$ is {\it $2$-hyperexpansive} (resp., {\em
completely hyperexpansive}) if $B_2(T) \Le 0$ (resp.,
$B_m(T) \Le 0$ for all positive integers $m$). The
notion of an $m$-isometric operator has been invented
by Agler (see \cite[p.\ 11]{Ag-0}). The concept of a
$2$-hyperexpansive operator goes back to Richter
\cite{R-0} (see also \cite[Remark~ 2]{At}). The notion
of a completely hyperexpansive operator has been
introduced by Athavale \cite{At}. It is well-known
that a $2$-isometry is $m$-isometric for every integer
$m\Ge 2$, and thus it is completely hyperexpansive
(see \cite[Paper I, \S 1]{Ag-St}).

The {\it Cauchy dual operator} $T'$ of a
left-invertible operator $T\in \B(\hh)$ is defined~
by\footnote{\;Note that left-invertibility of $T$
implies invertibility of $T^*T$ in the algebra
$\B(\hh)$.}
   \begin{align*}
T'=T(T^*T)^{-1}.
   \end{align*}
$T'$ is also called the {\em Cauchy dual} of $T$.
Recall that the range of a left-invertible operator is
always closed. It is easily seen that if $T$ is
left-invertible, then $T'$ is again left-invertible
and the following conditions hold{\em :}
   \begin{gather}  \label{tpr0}
(T')'=T,
   \\ \label{tpr}
T^*T'=I,
   \\ \label{tpr2}
T'T^* \text{ is the orthogonal projection of $\hh$
onto } T(\hh),
   \\  \label{tpr3}
T'^*T'=(T^*T)^{-1}.
   \end{gather}
The notion of the Cauchy dual operator has been
introduced and studied by Shimorin in the context of
the wandering subspace problem for Bergman-type
operators \cite{Sh}. The Cauchy dual technique has
been employed in \cite{Ch-0} to prove Berger-Shaw-type
theorems for $2$-hyperexpansive (or in Shimorin's
terminology, concave) operators. It is also important
to note that
   \begin{align} \label{2hypcon}
   \begin{minipage}{70ex}
{\em if $T\in \B(\hh)$ is a $2$-isometry $($or, more
generally, a $2$-hyperexpansive operator$)$, then $T$
is left-invertible and $T'$ is a contraction.}
   \end{minipage}
   \end{align}
Indeed, by \cite[Lemma 1]{R-0}, we have $\|Tf\| \Ge
\|f\|$ for all $f\in \hh$, which implies that $T$ is
left-invertible and $\|T^{\prime}\| = \|U |T|^{-1}\|
\Le 1$, where $T=U|T|$ is the polar decomposition of
$T$. The map that sends $T$ to $T'$ links
$2$-hyperexpansive operators to hyponormal ones (see
\cite[Sect.\ 5]{Sh-1} and \cite[Theorem~ 2.9]{Ch-0}).
Moreover, if $T$ is a $2$-hyperexpansive operator,
then $T^{\prime}$ is power hyponormal, i.e., all
positive integer powers $T'^n$ of $T'$ are hyponormal
(see \cite[Theorem~ 3.1]{Ch-1}). What is more
interesting, if $T$ is a completely hyperexpansive
unilateral weighted shift, then $T'$ is a subnormal
contraction (see \cite[Proposition~ 6]{At}). This
leads to the question originally posed in
\cite[Question 2.11]{Ch-0}: is the Cauchy dual of a
completely hyperexpansive operator a subnormal
contraction? Here we consider the following version of
this problem.
   \begin{opq*}
Is the Cauchy dual of a $2$-isometry a subnormal
contraction?
   \end{opq*}
In this paper we solve the above problem and
\cite[Question 2.11]{Ch-0} in the negative (see
Examples~ \ref{glowny} and \ref{przadj}). Since, by
\cite[Theorem~ 3.1]{Ch-1}, the Cauchy dual of a
$2$-isometry is power hyponormal, we get also examples
of non-subnormal power hyponormal operators (cf.\
\cite{Cu-Put0,Cu-Put}). The considerations in
\cite{A-C} related to the above problem enabled the
first two authors to solve negatively the problem of
subnormality of module tensor product of two subnormal
modules posed by Salinas in~ 1988.

Using \eqref{2hypcon} and \cite[Corollary]{Emb}, we
can reformulate the Cauchy dual subnormality problem
as follows: is the sequence $\{T^{\prime * n}
T'^n\}_{n=0}^{\infty}$ an operator Hausdorff moment
sequence for a $2$-isometry $T \in \B(\hh)$? Hence, it
is important to determine the exact form of the
operator Hausdorff moment sequence $\{T^{\prime * n}
T'^n\}_{n=0}^{\infty}$ if $T^{\prime}$ is subnormal.

It is worth mentioning that the question of
subnormality of $2$-isometric operators has a simple
solution. This is due to the following more general
result which is a direct consequence of \cite[Lemma
1]{R-0}: if $T\in \B(\hh)$ is a subnormal $($or, more
generally, a normaloid\/$)$ operator which is
$2$-hyperexpansive, then $T$ is an isometry.

For future reference, we explicitly state two
celebrated criteria for subnormality of bounded
operators that are due to Lambert and Agler,
respectively. It is worth pointing out that in view of
the Hausdorff moment theorem (see \cite[Theorem~
4.6.11]{B-C-R}), these two results are equivalent.
   \begin{theorem}[\mbox{\cite{lam}, \cite[Proposition~
2.3]{St1}}] \label{Lam}
   An operator $S\in \B(\hh)$ is subnormal if and only
if for every $f\in \hh,$ the sequence $\{\|S^n
f\|^2\}_{n=0}^{\infty}$ is a Stieltjes moment
sequence, i.e., there exists a positive Borel measure
$\mu_f$ on $[0,\infty)$ such that
   \begin{align*}
\|S^n f\|^2 = \int_{[0,\infty)} t^n d \mu_f(t), \quad
n=0,1,2, \ldots.
   \end{align*}
   \end{theorem}
   \begin{theorem}[\mbox{\cite[Theorem~ 3.1]{Ag}}] \label{Ag-Emb}
   An operator $S\in \B(\hh)$ is a subnormal
contraction if and only if $B_m(S) \Ge 0$ for all $m=
1,2,3, \ldots$.
   \end{theorem}
It is also of some interest to reveal the following
intimate connection of the Cauchy dual subnormality
problem to the theory of Toeplitz matrices.
   \begin{proposition}
Let $T \in \B(\hh)$ be a $2$-isometry such that for
all $h \in \hh$ and $k\in\{0,1,2, \ldots\}$, the
Toeplitz matrix $[\inp{L_{j-i}h}{h}]_{i,j=0}^{k}$ is
positive definite, where
   \begin{align*}
L_n :=
   \begin{cases}
(T^{n}T'^{n})^* & \text{if $n$ is a nonnegative
integer},
   \\
T^{|n|}T'^{|n|} & \text{if $n$ is a negative integer}.
   \end{cases}
   \end{align*}
Then the Cauchy dual $T'$ of $T$ is subnormal.
   \end{proposition}
   \begin{proof} It follows from \eqref{tpr} that
   \begin{align*}
(T'^n)^* = (T^{n}T'^{n})^* T'^n, \quad n =
0,1,2,\ldots,
   \end{align*}
which together with \cite[Theorem~ 1]{T} completes the
proof.
   \end{proof}
The remainder of the paper is organized as follows. In
Section~ \ref{Sec3} we give a purely algebraic
characterization of $2$-isometries satisfying the
kernel condition (see Lemma~ \ref{char-1}) and provide
a model for such operators built on operator valued
unilateral weighted shifts (see Theorem~ \ref{model}).
We also give a brief discussion of the case of cyclic
analytic $2$-isometries satisfying the kernel
condition (see Proposition~ \ref{2isokk}). In Section~
\ref{Sec6}, we prove that the Cauchy dual $T^{\prime}$
of a $2$-isometry $T$ satisfying the kernel condition
is a subnormal contraction (see Theorem~
\ref{cdsubn}). Theorem~ \ref{cdsubn} is not true for
$m$-isometries for $m\Ge 3$ (see Example~
\ref{treiso}). Section~ \ref{Sec7} deals with
quasi-Brownian isometries, a subclass of
$2$-isometries, which generalize Brownian isometries
introduced by Agler and Stankus in \cite{Ag-St}. Using
a block operator model for this class of operators
given in \cite{Maj}, we show that the Cauchy dual
$T^\prime$ of a quasi-Brownian isometry $T$ is a
subnormal contraction (see Theorem~ \ref{BrownianG}).
We also show that a quasi-Brownian isometry satisfying
the kernel condition must be an isometry (see
Corollary~ \ref{binkc}). In Section~ \ref{Sec8}, we
collect selected properties of weighted shifts on
directed trees that we need for further
considerations. This class of operators was introduced
in \cite{JJS} and intensively studied since then (see
e.g., \cite{JJS-1,Geh,C-T,B-D-P,BJJS-aim,M-A,K-L-P}).
We show, among other things, that the Cauchy dual of a
left-invertible weighted shift on a directed tree is
still a weighted shift on a directed tree (see Lemma~
\ref{cisws}). We also prove that $2$-isometric
weighted shifts on rootless directed trees with
nonzero weights that satisfy the kernel condition and
Brownian isometric weighted shifts on rooted directed
trees are automatically isometric (see Propositions~
\ref{bilat} and \ref{briai}). In Section~ \ref{Sec10},
we prove a reconstruction theorem stating that if
$\slam$ is a $2$-isometric weighted shift on a rooted
directed tree with nonzero weights that satisfies the
perturbed kernel condition \eqref{hypok} for some
positive integer $k$, then $\slam^\prime$ is subnormal
if and only if $\slam$ satisfies the kernel condition
(see Theorem~ \ref{main2}). The reconstruction theorem
enables us to answer the Cauchy dual subnormality
problem in the negative (see Example~ \ref{glowny}).
In Section~ \ref{Sec11}, we investigate the Cauchy
dual subnormality problem for adjacency operators of
directed trees. This class of operators plays an
important role in graph theory (see
\cite{b-m-s,f-s-w,JJS}). We show that if a rooted
directed tree satisfies certain degree constraints and
the corresponding adjacency operator $\slamj$ is a
$2$-isometry, then the Cauchy dual operator
$\slamj^\prime$ of $\slamj$ is subnormal (Theorem~
\ref{constant-t}). This enables us to construct a
$2$-isometric adjacency operator $\slamj$ of a
directed tree, which does not satisfy the kernel
condition, which is not a quasi-Brownian isometry, but
which has the property that $\slamj^\prime$ is a
subnormal contraction (see Example~ \ref{nbnkcsub}).
Finally, we show that the Cauchy dual subnormality
problem has a negative solution in the class of
adjacency operators of directed trees (see Example~
\ref{przadj}).

Now we fix notation and terminology. Let $\zbb$,
$\rbb$ and $\C$ stand for the sets of integers, real
numbers and complex numbers, respectively. Denote by
$\nbb$, $\zbb_+$ and $\rbb_+$ the sets of positive
integers, nonnegative integers and nonnegative real
numbers, respectively. Given a set $X$, we write
$\chi_{\varDelta}$ for the characteristic function of
a subset $\varDelta$ of $X$. The $\sigma$-algebra of
all Borel subsets of a topological space $X$ is
denoted by $\borel{X}$. For $a\in \rbb,$ $\delta_a$
stands for the Borel probability measure on $\rbb$
supported on $\{a\}$. In this paper, Hilbert spaces
are assumed to be complex and operators are assumed to
be linear. Let $\hh$ be a Hilbert space. As usual, we
denote by $\dim \hh$ the orthogonal dimension of
$\hh$. If $f \in \hh$, then $\langle f \rangle$ stands
for the linear span of the singleton of $f$. Given
another Hilbert space $\kk$, we denote by
$\B(\hh,\kk)$ the Banach space of all bounded
operators from $\hh$ to $\kk$. The kernel and the
range of an operator $T \in \B(\hh,\kk)$ are denoted
by $\ker T$ and $\ran T$, respectively. We abbreviate
$\B(\hh,\hh)$ to $\B(\hh)$ and regard $\B(\hh)$ as a
$C^*$-algebra. The identity operator on $\hh$ is
denoted by $I_\hh$ (or simply by $I$ if no ambiguity
arises). We write $\sigma(T)$ for the spectrum of
$T\in \B(\hh)$. If $S$ and $T$ are Hilbert space
operators which are unitarily equivalent, then we
write $S \cong T$.

Following \cite{R-1}, we say that an operator $T\in
\B(\hh)$ is {\em analytic} if $\bigcap_{n=1}^{\infty}
T^n(\hh)=\{0\}$. An operator $T\in\B(\hh)$ is said to
be {\em completely non-unitary} (resp., {\em pure}) if
there is no nonzero reducing closed vector subspace
$\mathcal L$ of $\hh$ such that the restriction
$T|_{\mathcal L}$ of $T$ to $\mathcal L$ is a unitary
(resp., a normal\/) operator. Clearly, every analytic
operator is completely non-unitary. Recall that any
operator $T\in \B(\hh)$ has a unique orthogonal
decomposition $T=N\oplus R$ such that $N$ is a normal
operator and $R$ is a pure operator (see
\cite[Corollary~ 1.3]{Mo}). We shall refer to $N$ and
$R$ as the {\em normal} and {\em pure} parts of $T$,
respectively.
   \section{\label{Sec3}A model for $2$-isometries
satisfying the kernel condition}
   The goal of this section is to show that a
non-unitary $2$-isometry satisfying the kernel
condition is unitarily equivalent to an orthogonal sum
of a unitary operator and an operator valued
unilateral weighted shift (see Theorem~ \ref{model}).

We say that an operator $T\in \B(\hh)$ satisfies the
{\em kernel condition} if
   \begin{align}  \label{kc}
T^*T (\ker T^*) \subseteq \ker T^*.
   \end{align}
By the square root lemma (see \cite[Theorem~
2.4.4]{Sim-4}), \eqref{kc} holds if and only if
   \begin{align*}
|T| (\ker T^*) \subseteq \ker T^*.
   \end{align*}
It is easily seen that any positive integral power of
a unilateral weighted shift satisfies the kernel
condition (see also the proof of Corollary~
\ref{nthpow}). It is a routine matter to verify that
weighted translation semigroups studied by Embry and
Lambert in \cite{EL} consist of operators satisfying
the kernel condition. Other examples of operators
satisfying the kernel condition will appear in this
paper when solving the Cauchy dual subnormality
problem.

The kernel condition is preserved by the operation of
taking the Cauchy dual.
   \begin{proposition} \label{kc-Cu}
Let $T\in \B(\hh)$ be a left-invertible operator. Then
the following conditions are equivalent{\em :}
   \begin{enumerate}
   \item[(i)] $T$ satisfies the kernel condition,
   \item[(ii)] $T^*T(\ker T^*) = \ker T^*,$
   \item[(iii)] $T'$ satisfies the kernel condition.
   \end{enumerate}
   \end{proposition}
   \begin{proof}
The equivalence (i)$\Leftrightarrow$(ii) is a
consequence of the following fact.
   \begin{align} \label{factaa}
   \begin{minipage}{72ex}
{\em If $A\in \B(\hh)$ is a selfadjoint operator which
is invertible in $\B(\hh)$ and $\mathcal L$ is a
closed vector subspace of $\hh$ which is invariant for
$A$, then $A(\mathcal L)=\mathcal L.$}
   \end{minipage}
   \end{align}
This together with \eqref{tpr0}, \eqref{tpr3} and the
equation $\ker T'^* = \ker T^*$ yields
(ii)$\Leftrightarrow$(iii).
   \end{proof}
The next result, whose proof is left to the reader,
shows that under some circumstances the Cauchy dual of
a restriction of a left-invertible operator to its
invariant subspace is equal to the restriction of the
Cauchy dual operator.
   \begin{proposition} \label{restrcd}
Suppose that $T \in \B(\hh)$ is a left-invertible
operator and $\mathcal L$ is a closed vector subspace
of $\hh$ such that $T(\mathcal L) \subseteq \mathcal
L$ and $T^*T(\mathcal L) \subseteq \mathcal L$. Then
$T|_{\mathcal L}$ is left-invertible,
$T^{\prime}(\mathcal L) \subseteq \mathcal L$,
$T^{\prime *}T^{\prime}(\mathcal L) \subseteq \mathcal
L$ and $(T|_{\mathcal L})^{\prime} =
T^{\prime}|_{\mathcal L}$. In particular, if $\mathcal
L$ reduces $T$, then $\mathcal L$ and $\mathcal
L^{\perp}$ reduce $T'$ and $T^{\prime}= (T|_{\mathcal
L})^{\prime} \oplus (T|_{\mathcal
L^{\perp}})^{\prime}$.
   \end{proposition}
Note that in general the assumptions of Proposition~
\ref{restrcd} do not imply that $\mathcal L$ reduces
$T$. Indeed, it is enough to consider an isometric
unilateral shift $V$ of multiplicity $1$ and any of
its nontrivial closed invariant vector subspaces, say
$\mathcal L$. Then $\mathcal L$ does not reduce $V$.
This is due to the fact that $V$ is irreducible, i.e.,
there is no nontrivial closed vector subspace of $\hh$
which reduces $V$ (see \cite[Corollary~ 2, p.\
63]{Shi}). Recall that the Beurling theorem completely
describes the lattice of all closed vector subspaces
of $\hh$ which are invariant for $V$ (see
\cite[Theorem~ 17.21]{Rud}).

   Now we give some purely algebraic characterizations
of $2$-isometries satisfying the kernel condition.
   \begin{lemma} \label{char-1}
Let $T \in \B(\hh)$ be a left-invertible operator.
Then the following conditions are equivalent{\em :}
   \begin{enumerate}
   \item[(i)] $T$ is a $2$-isometry  such that $T^*T (\ker T^*)
\subseteq \ker T^*,$
   \item[(ii)] $T$ is a $2$-isometry such that $T^*T(\ran T)
\subseteq \ran T,$
   \item[(iii)] $T' -2T + T^*T^2 =0,$
   \item[(iv)] $(T^*T^2 T^*-2TT^* + I)T=0,$
   \item[(v)] $T'(T^*T-I)=(T^*T-I)T.$
   \end{enumerate}
   \end{lemma}
   Observe that, by \eqref{factaa}, Lemma~
\ref{char-1} remains true if inclusions appearing in
(i) and (ii) are replaced by equalities.
   \begin{proof}[Proof of Lemma~ \ref{char-1}]
Set $\triangle_T=T^*T-I$ and $\nabla_T=T' T'^*-2I
+T^*T$. Since, by \eqref{tpr}, $T'^*T=I$ we see that
$\nabla_T T = T' -2T + T^*T^2$ and $T^* \nabla_T T =
I-2T^*T+{T^*}^2T^2$. This implies that
   \begin{enumerate}
   \item[(a)] the equation (iii) is equivalent to $\nabla_T T=0,$
   \item[(b)] $T$ is a $2$-isometry if and only if $T^*\nabla_T T =0,$
   \item[(c)] $T$ is a $2$-isometry if and only if
$T^*\triangle_T T=\triangle_T.$
   \end{enumerate}

(i)$\Leftrightarrow$(ii) This is obvious because $\ran
T$ is closed.

(ii)$\Rightarrow$(iii) Assume that (ii) holds. Since
by assumption $\ran(\nabla_T T) \subseteq \ran{T}$
and, by \eqref{tpr2}, $P:=T'T^*$ is the orthogonal
projection of $\hh$ onto $\ran T$, we get
   \begin{align*}
\nabla_T T =P\nabla_T T= T'(T^*\nabla_T T)
\overset{\rm (b)}= 0.
   \end{align*}
Hence, by (a), the equality (iii) holds.

(iii)$\Leftrightarrow$(iv) This can be deduced from
the definition of $T'.$

(iii)$\Rightarrow$(v) If (iii) holds, then
   \begin{align*}
T'(T^*T-I) = T - T' \overset{\mathrm{(iii)}}= T^*T^2 -
T = (T^*T-I)T,
   \end{align*}
which gives (v).

(v)$\Rightarrow$(ii) Suppose (v) holds. Note that
   \begin{align*}
\triangle_T \overset{\eqref{tpr}} = T^*T'\triangle_T
\overset{\mathrm{(v)}} = T^*\triangle_TT,
   \end{align*}
and thus, by (c), $T$ is a $2$-isometry. It suffices
to check that $\triangle_T(\mbox{ran} ~T) \subseteq
\mbox{ran} ~T$. As above $P=T'T^*$. Since $\ran(T')
\subseteq \ran ~ T$, we have
   \begin{align*}
T'\triangle_T \overset{\eqref{tpr2}} = PT'\triangle_T
\overset{\mathrm{(v)}}= P\triangle_TT.
   \end{align*}
This implies that
   \begin{align*}
(I-P)\triangle_TT=\triangle_T T - T'\triangle_T
\overset{\mathrm{(v)}} =0.
   \end{align*}
Hence $P\triangle_TT=\triangle_TT$ and thus
$\triangle_T(\mbox{ran} ~T) \subseteq \mbox{ran} ~T$.
This completes the proof.
   \end{proof}
Below, we collect some properties (whose verifications
are left to the reader) of the sequence
$\{\xi_n\}_{n=0}^{\infty}$ of self-maps of the
interval $[1,\infty)$ given by
   \begin{align} \label{xin}
\xi_n(x) = \sqrt{\frac{1+ (n+1)(x^2-1)}{1+ n(x^2-1)}},
\quad x \in [1,\infty), \, n\in \zbb_+.
   \end{align}
   \begin{lemma} \label{xin11}
The sequence $\{\xi_n\}_{n=0}^\infty$ given by
   \eqref{xin} has the following properties{\em :}
\begin{enumerate}
   \item[(i)] $\xi_0(x)=x$ for all $x \in [1,\infty)$,
   \item[(ii)] $\xi_{m+n}(x) = (\xi_m \circ \xi_n)(x)$
for all $x \in [1,\infty)$ and $m,n\in \zbb_+$,
   \item[(iii)] $\xi_n(1)=1$ for all $n\in \zbb_+$,
   \item[(iv)] $\xi_{n}(x) > \xi_{n+1}(x) > 1$ for all
$x \in (1,\infty)$ and $n\in \zbb_+$.
   \end{enumerate}
   \end{lemma}
Before stating the main result of this section, we
recall the definition of an operator valued unilateral
weighted shift. Let $\mathcal M$ be a
\underline{nonzero} Hilbert space. Denote by
$\ell^2_{\mathcal M}$ the Hilbert space of all vector
sequences $\{h_n\}_{n=0}^{\infty} \subseteq \mathcal
M$ such that $\sum_{n=0}^{\infty} \|h_n\|^2 < \infty$
equipped with the standard inner product
   \begin{align*}
\big\langle \{g_n\}_{n=0}^{\infty},
\{h_n\}_{n=0}^{\infty}\big\rangle =
\sum_{n=0}^{\infty} \langle g_n, h_n \rangle, \quad
\{g_n\}_{n=0}^{\infty}, \{h_n\}_{n=0}^{\infty} \in
\ell^2_{\mathcal M}.
   \end{align*}
If $\{W_n\}_{n=0}^{\infty} \subseteq \B(\mathcal M)$
is a uniformly bounded sequence of operators, then the
operator $W \in \B(\ell^2_{\mathcal M})$ defined by
   \begin{align} \label{opvuws}
W(h_0, h_1, \ldots) = (0, W_0 h_0, W_1 h_1, \ldots),
\quad (h_0, h_1, \ldots) \in \ell^2_{\mathcal M},
   \end{align}
is called an {\em operator valued unilateral weighted
shift} with weights $\{W_n\}_{n=0}^{\infty}$. If each
weight $W_n$ of $W$ is an invertible (resp., a
positive) element of the $C^*$-algebra $\B(\mathcal
M)$, then we say that $W$ is an operator valued
unilateral weighted shift with {\em invertible}
(resp., {\em positive}) weights. Putting $\mathcal
M=\C$, we arrive at the well-known notion of a
unilateral weighted shift in $\ell^2$.

Let $W$ be as in \eqref{opvuws}. It is easy to verify
that
   \begin{align} \label{aopws}
W^*(h_0, h_1, \ldots) &= (W_0^*h_1, W_1^*h_2, \ldots),
\quad (h_0, h_1, \ldots) \in \ell^2_{\mathcal M},
   \\   \label{aopws2}
W^*W(h_0, h_1, \ldots) &= (W_0^*W_0h_0, W_1^*W_1h_1,
\ldots), \quad (h_0, h_1, \ldots) \in \ell^2_{\mathcal
M}.
   \end{align}
Given integers $m\Ge n \Ge 0$, we set (see \cite[p.\
409]{Ja-3})
   \begin{align} \label{wmndef}
W_{m,n} =
   \begin{cases}
W_{m-1} W_{m-2} \, \cdots \, W_n & \text{if } m>n,
   \\[.5ex]
I & \text{if } m=n.
   \end{cases}
   \end{align}

   Now we characterize non-unitary $2$-isometric
operators satisfying the kernel condition. In what
follows, by a unitary operator, we mean a unitary
operator $U\in \B(\kk)$, where $\kk$ is an arbitrary
Hilbert space; the case $\kk=\{0\}$ is not excluded.
   \begin{theorem} \label{model}
If $T\in \B(\hh)$ is a non-unitary $2$-isometry, then
the following conditions are equivalent{\em :}
   \begin{enumerate}
   \item[(i)] $T$ satisfies the kernel condition,
   \item[(ii)] $T(\ker T^*) \perp T^n(\ker T^*)$ for
every integer $n \Ge 2,$
   \item[(iii)] the spaces $\{T^n (\ker T^*)\}_{n=0}^{\infty}$
are mutually orthogonal,
   \item[(iv)] $T\cong U \oplus W$, where $U$ is a unitary operator
and $W$ is an operator valued unilateral weighted
shift with invertible positive weights,
   \item[(v)] $T\cong U \oplus W,$ where
$U$ is a unitary operator and $W$ is an operator
valued unilateral weighted shift on $\ell^2_{\mathcal
M}$ with weights $\{W_n\}_{n=0}^{\infty}$ defined by
   \begin{align} \label{wagi}
   \left.
   \begin{gathered} W_n = \int_{[1,\infty)}
\xi_n(x) E(d x), \quad n \in \zbb_+,
   \\
   \begin{minipage}{62ex}
where $E$ is a compactly supported $\B(\mathcal
M)$-valued Borel spectral measure on the iterval
$[1,\infty)$ and $\{\xi_n\}_{n=0}^{\infty}$ is defined
by \eqref{xin}.
   \end{minipage}
   \end{gathered}
   \; \right\}
   \end{align}
   \end{enumerate}
Moreover, the following hold{\em :}
   \begin{enumerate}
   \item[(a)] if $T\in \B(\hh)$ is a non-unitary
$2$-isometry that satisfies {\em (i)}, then $\hh_u:=
\bigcap_{n = 1}^{\infty} T^n (\hh)$ is a closed vector
subspace of $\hh$ reducing $T$ to a unitary operator
and $T|_{\hh_u^{\perp}} \cong W$, where $W$ is an
operator valued unilateral weighted shift on
$\ell^2_{\mathcal M}$ with weights
$\{W_n\}_{n=0}^{\infty}$ given by \eqref{wagi} and
$\dim \mathcal M = \dim \ker T^*,$
   \item[(b)] if $U$
is a unitary operator and $W$ is an operator valued
unilateral weighted shift on $\ell^2_{\mathcal M}$
with weights\footnote{\;Note that, in view of
\eqref{unifogr}, the sequence $\{W_n\}_{n=0}^{\infty}
\subseteq \B(\mathcal M)$ defined by \eqref{wagi} is
uniformly bounded, and consequently $W\in
\B(\ell^2_{\mathcal M})$.} $\{W_n\}_{n=0}^{\infty}$
defined by \eqref{wagi}, then $T:=U \oplus W$ is a
non-unitary $2$-isometry that satisfies {\em (i)}, $U$
is the normal part of $T$, $W$ is the pure part of $T$
and $\ker W^*=\mathcal M \oplus \{0\} \oplus \{0\}
\oplus \ldots.$
   \end{enumerate}
   \end{theorem}
   \begin{proof}
Assume that $T$ is a non-unitary $2$-isometry.

(i)$\Rightarrow$(ii) Note that for every integer $n
\Ge 2$,
   \begin{align*}
\langle T f, T^n g \rangle = \langle T^*(T^*T f),
T^{n-2} g \rangle = 0, \qquad f,g \in \ker T^*.
   \end{align*}

(ii)$\Rightarrow$(iii) It suffices to show that for
every integer $n\Ge 0,$
   \begin{align}  \label{tjk}
T^j (\ker T^*) \perp T^k (\ker T^*), \quad k \in
\{j+1, j+2,\ldots\}, \, j \in \{0, \ldots, n\}.
   \end{align}
We use induction on $n$. The cases $n=0$ and $n=1$ are
obvious. Suppose \eqref{tjk} holds for a fixed $n\Ge
1$. Since $I-2T^*T + {T^*}^2T^2 =0$ yields
   \begin{align} \label{tjk-2}
T^{*(n-1)}T^{n-1} -2T^{*n} T^n + T^{*(n+1)}T^{n+1} =
0,
   \end{align}
we deduce that for every integer $k \Ge n+2$,
   \allowdisplaybreaks
   \begin{align*}
\langle T^{n+1}f, T^k g\rangle & = \langle
T^{*(n+1)}T^{n+1}f, T^{k-(n+1)} g\rangle
   \\
& \hspace{-.3ex}\overset{\eqref{tjk-2}}= 2 \langle
T^nf, T^{k-1} g\rangle - \langle T^{n-1}f, T^{k-2}
g\rangle
   \\
& \hspace{-.3ex} \overset{\eqref{tjk}} = 0, \quad f,g
\in \ker T^*,
   \end{align*}
   which completes the induction argument and gives
(iv).

(iii)$\Rightarrow$(v) By \cite[Theorem~ 3.6]{Sh}
(which is also valid for inseparable Hilbert spaces),
$\hh_u$ is a closed vector subspace of $\hh$ that
reduces $T$ to the unitary operator $U$ and
$\hh_u^{\perp} = \bigvee_{n=0}^{\infty} T^n(\ker
T^*)$. Since $T$ is non-unitary, $\hh_u^{\perp} \neq
\{0\}$ and consequently $\ker T^* \neq \{0\}$. Hence,
by the injectivity of $T$, $\mathcal M_n:=T^n(\ker
T^*) \neq \{0\}$ for all $n\in \zbb_+$. Clearly,
$A:=T|_{\hh_u^{\perp}}$ is a $2$-isometry. Since the
operator $T$ is bounded from below, we deduce that for
every $n\in \zbb_+$, $\mathcal M_n$ is a closed vector
subspace of $\hh$ and $\varLambda_n:= T|_{\mathcal
M_n}\colon \mathcal M_n \to \mathcal M_{n+1}$ is a
linear homeomorphism. Therefore, by \cite[Problem
56]{Hal}, for every $n\in \zbb_+$, the Hilbert spaces
$\mathcal M_n$ and $\mathcal M_0$ are unitarily
equivalent. Let, for $n\in \zbb_+$, $V_n\colon
\mathcal M_n \to \mathcal M_0$ be any unitary
isomorphism. Since $\hh_u^{\perp} =
\bigoplus_{n=0}^{\infty} \mathcal M_n$, we can define
the unitary isomorphism $V\colon \hh_u^{\perp} \to
\ell^2_{\mathcal M_0}$ by
   \begin{align*}
V(h_0 \oplus h_1 \oplus \ldots) = (V_0h_0, V_1h_1,
\ldots), \quad h_0 \oplus h_1 \oplus \ldots \in
\hh_u^{\perp}.
   \end{align*}
Let $S\in \B(\ell^2_{\mathcal M_0})$ be the operator
valued unilateral weighted shift with (uniformly
bounded) weights $\{V_{n+1} \varLambda_n
V_n^{-1}\}_{n=0}^{\infty} \subseteq \B(\mathcal M_0)$.
Then
   \allowdisplaybreaks
   \begin{align*}
VA(h_0 \oplus h_1 \oplus \ldots) & = V(0 \oplus
\varLambda_0 h_0 \oplus \varLambda_1 h_1 \oplus
\ldots)
   \\
& = (0, (V_1\varLambda_0 V_0^{-1}) V_0h_0,
(V_2\varLambda_1 V_1^{-1})V_1h_1, \ldots)
   \\
& = SV(h_0 \oplus h_1 \oplus \ldots), \quad h_0 \oplus
h_1 \oplus \ldots \in \hh_u^{\perp},
   \end{align*}
which means that $A$ is unitarily equivalent to $S$.
Since the weights of the $2$-isometry $S$ are
invertible in $\B(\mathcal M_0)$, we infer from
\cite[Corollary~ 2.3]{Ja-3} that $S$ is unitarily
equivalent to a $2$-isometric operator valued
unilateral weighted shift $W$ on $\ell^2_{\mathcal
M_0}$ with invertible weights $\{W_n\}_{n=0}^{\infty}
\subseteq \B(\mathcal M_0)$ such that $W_n \, \cdots\,
W_0 \Ge 0$ for all integers $n \Ge 0$. In turn, by
\cite[Lemma 1]{R-0}, $\|Wh\| \Ge \|h\|$ for all $h\in
\ell^2_{\mathcal M_0}$, which yields
   \begin{align*}
\|W_0 h_0\| = \|(0, W_0 h_0, 0, \ldots)\| = \|W(h_0,
0, \ldots)\| \Ge \|h_0\|, \quad h_0 \in \mathcal M_0.
   \end{align*}
Hence $W_0 \Ge I$. This combined with the proof of
\cite[Theorem~ 3.3]{Ja-3} implies that
   \begin{align*}
W_n = \int_{[1,\|W_0\|]} \hat\xi_n(x) E(d x), \quad
n\in \zbb_+,
   \end{align*}
where $E$ is the spectral measure of $W_0$ and
$\{\hat\xi_n\}_{n=0}^{\infty}$ is the sequence of
self-maps of the interval $[1,\infty)$ defined
recursively by
   \begin{align*}
\text{$\hat\xi_0(x) = x$ and $\hat\xi_{n+1}(x) =
\sqrt{\frac{2\hat\xi_{n}^2(x)-1}{\hat\xi_{n}^2(x)}}$
for all $x \in [1, \infty)$ and $n\in \zbb_+.$}
   \end{align*}
Using induction, one can show that $\hat\xi_n = \xi_n$
for all $n\in \zbb_+$, which gives (v).

The implication (v)$\Rightarrow$(iv) is obvious.

(iv)$\Rightarrow$(i) Let $W$ be as in \eqref{opvuws}.
Since the weights of $W$ are invertible, we infer from
\eqref{aopws} that $\ker W^* = \mathcal M \oplus \{0\}
\oplus \{0\} \oplus \ldots.$ This combined with
\eqref{aopws2} yields $W^*W(\ker W^*) = \ker W^*$,
which gives (i).

Now we turn to the proof of the ``moreover'' part.

(a) This has already been done in the proof of the
implication (iii)$\Rightarrow$(v).

(b) First, we show that $U$ and $W$ are the normal and
pure parts of $T$, respectively. Denote by $\kk$ the
Hilbert space in which the unitary operator $U$ acts.
Since $W$ is an operator valued unilateral weighted
shift with invertible weights, we infer from
\eqref{opvuws} and \eqref{aopws} that
   \begin{align*}
\kk \oplus \{0\} \subseteq \bigcap_{k=1}^{\infty}
\bigcap_{l=1}^{\infty} \ker (T^{*k}T^l - T^l T^{*k})
\subseteq \bigcap_{k=1}^{\infty} \ker (T^{*k}T^k - T^k
T^{*k}) \subseteq \kk \oplus \{0\}.
   \end{align*}
This, together with \cite[Corollary~ 1.3]{Mo}, proves
our claim.

Arguing as in the proof of the implication
(iv)$\Rightarrow$(i), we verify that $\ker W^* =
\mathcal M \oplus \{0\} \oplus \{0\} \oplus \ldots$
and $T$ satisfies (i). Therefore, it remains to show
that $W$, and consequently $T$, are $2$-isometries.
Since, by the Stone-von Neumann calculus for
selfadjoint operators, $\|W_n\| = \sup_{x\in \supp E}
\xi_n(x)$ for every $n\Ge 0$, where $\supp E$ denotes
the closed support of $E$, we get (see Lemma~
\ref{xin11})
   \begin{align} \label{unifogr}
\sup_{n\Ge 0} \|W_n\| \Le \sup_{n\Ge 0} \, \sup_{x \in
[1,\eta]} \xi_n(x) = \sup_{x \in [1,\eta]} \,
\sup_{n\Ge 0} \xi_n(x) = \sup_{x \in [1,\eta]}
\xi_0(x) = \eta,
   \end{align}
where $\eta:=\sup(\supp{E})$. This implies that $W \in
\B(\ell^2_{\mathcal M})$ and $\|W\| \Le \eta$ (see
\cite[p.\ 408]{Ja-3}). By \eqref{wagi} and the
multiplicativity of spectral integral, we have
(consult \eqref{wmndef})
   \begin{align} \label{wmn}
W_{m,n} = \int_{[1,\eta]}
\sqrt{\frac{1+m(x^2-1)}{1+n(x^2-1)}} \, E(d x), \quad
m \Ge n \Ge 0.
   \end{align}
This implies that for all integers $s\Ge 0$,
   \allowdisplaybreaks
   \begin{multline*}
\sum_{j=0}^2 (-1)^j \binom 2 j \|W_{s+j,s}f\|^2
   \\
= \int_{[1,\eta]} \frac{\sum_{j=0}^2 (-1)^j \binom 2 j
(1+(s+j)(x^2-1))}{1+s(x^2-1)} \, \langle E(d x)f,f
\rangle = 0, \quad f \in \mathcal M.
   \end{multline*}
Hence, in view of \cite[Proposition~ 2.5(i)]{Ja-3},
$W$ is a $2$-isometry. This completes the proof.
   \end{proof}
   \begin{corollary} \label{Zak1}
Let $T\in \B(\hh)$ be a $2$-isometry that satisfies
the kernel condition. Then the following statements
are equivalent{\em :}
   \begin{enumerate}
   \item[(i)] $T$ is analytic,
   \item[(ii)] $T$ is completely non-unitary,
   \item[(iii)] $T$ is pure,
   \item[(iv)] $T$ is unitarily
equivalent to an operator valued unilateral weighted
shift $W$ on $\ell^2_{\mathcal M}$ with weights
$\{W_n\}_{n=0}^{\infty}$ defined by \eqref{wagi},
where $\mathcal M = \ker T^*$.
   \end{enumerate}
   \end{corollary}
   \begin{corollary}
Let $U$, $W,$ $E$ and $T$ be as in Theorem~ {\em
\ref{model}(b)}. Then $T$ is an isometry if and only
if $\supp E = \{1\}.$
   \end{corollary}
   \begin{proof}
If $T$ is an isometry, then $W$ is an isometry. This
together with \eqref{aopws2} and \eqref{wagi} implies
that $W_0$ is positive and unitary, and so
   \begin{align*}
\supp E = \sigma(W_0)=\{1\}.
   \end{align*}
The reverse implication is obvious because, due to
\eqref{wagi} and Lemma~ \ref{xin11}(iii),
$W_n=I_{\mathcal M}$ for all $n\in \zbb_+$.
   \end{proof}
It is easily seen that if an operator $T\in \B(\hh)$
is such $\ker(T^*T-I)=\{0\},$ then $T$ is completely
non-unitary. A converse to this implication is not
true in general even in the case of non-isometric
$2$-isometries satisfying the kernel condition (use
\eqref{aopws2} and Corollary~ \ref{Zak1} with $\dim
\mathcal M \Ge 2$).

We conclude this section by describing cyclic analytic
$2$-isometries satisfying the kernel condition. The
description itself relies on Richter's model for
cyclic analytic $2$-isometries. Namely, by
\cite[Theorem~ 5.1]{R-1}, a cyclic analytic
$2$-isometry is unitarily equivalent to the operator
$M_{z,\mu}$ of multiplication by the coordinate
function $z$ on a Dirichlet-type space $\mathcal
D(\mu)$, where $\mu$ is a finite positive Borel
measure on the interval $[0,2\pi)$ (which can be
identified with the finite positive Borel measure on
the unit circle \mbox{$\mathbb T = \{z\in \C \colon
|z|=1\}$).} The Hilbert space $\mathcal D(\mu)$
consists of all analytic function $f$ on the open unit
disc $\mathbb D = \{z\in \C \colon |z|< 1\}$ such that
   \begin{align*}
\int_{\mathbb D} |f'(z)|^2 \varphi_{\mu}(z) d A(z) <
\infty,
   \end{align*}
   where $f'$ stands for the derivative of $f$, $A$
denotes the normalized Lebesgue area measure on
$\mathbb D$ and $\varphi_{\mu}$ is the positive
harmonic function on $\mathbb D$ defined by
   \begin{align} \label{phimu1}
\varphi_{\mu}(z) = \frac{1}{2\pi} \int_{[0,2\pi)}
\frac{1-|z|^2}{|e^{it}-z|^2} d \mu(t), \quad z \in
\mathbb D.
   \end{align}
The inner product $\langle\,\cdot\,,
\mbox{-}\rangle_{\mu}$ of $\mathcal D(\mu)$ is given
by
   \begin{align} \label{phimu2}
\langle f, g \rangle_{\mu} = \langle f,g \rangle_{H^2}
+ \int_{\mathbb D} f'(z) \overline{g'(z)}
\varphi_{\mu}(z) d A(z), \quad f,g \in \mathcal
D(\mu),
   \end{align}
where $\langle \, \cdot\, ,\mbox{-} \rangle_{H^2}$
stands for the inner product of the Hardy space $H^2$.
The induced norm of $\mathcal D(\mu)$ is denoted by
$\|\cdot\|_{\mu}$. In this model, $2$-isometries
satisfying the kernel condition can be described as
follows\footnote{\;This answers the question of Eva A.
Gallardo-Guti\'{e}rrez asked during {\em Workshop on
Operator Theory, Complex Analysis and Applications}
2016, Coimbra, Portugal, June 21-24.}.
   \begin{proposition}\label{2isokk}
Under the above assumptions, $M_{z,\mu}$ satisfies the
kernel condition if and only if $\mu=\alpha \, m$ for
some $\alpha \in \rbb_+,$ where $m$ is the Lebesgue
measure on $[0,2\pi).$
   \end{proposition}
   \begin{proof}
Suppose $\mu=\alpha m$ for some $\alpha\in \rbb_+$.
Since the geometric series expansion of
$(1-re^{it})^{-1}$ is uniformly convergent with
respect to $t$ in $\rbb$, we infer from \eqref{phimu1}
that
   \begin{align*}
\varphi_{\mu}(re^{i\vartheta}) = \alpha, \quad
\vartheta \in \rbb, \, r \in [0,1).
   \end{align*}
This, together with \eqref{phimu2} and
\cite[Corollary~ 3.8(d)]{R-1}, implies that the
sequence $\{e_n\}_{n=0}^{\infty}$ defined by
   \begin{align*}
e_n(z)=\frac{1}{\sqrt{1+ n \alpha}} z^n, \quad z\in
\mathbb D, \, n\in \zbb_+,
   \end{align*}
is an orthonormal basis of $\mathcal D(\mu)$. Since
$M_{z,\mu} e_n = \sqrt{\frac{1+ (n+1)\alpha}{1+
n\alpha}} \, e_{n+1}$ for all $n\in \zbb_+$, we deduce
that $M_{z,\mu}$ is unitarily equivalent to the
unilateral weighted shift with weights
$\{\xi_n(\lambda)\}_{n=0}^{\infty}$, where
$\lambda:=\sqrt{1+\alpha}$. As a consequence,
$M_{z,\mu}$ satisfies the kernel condition.

Suppose now that $M_{z,\mu}$ is the operator of
multiplication by the coordinate function $z$ on
$\mathcal D(\mu)$ satisfying the kernel condition,
where $\mu$ is a finite positive Borel measure on
$[0,2\pi)$. Since $M_{z,\mu}$ is an analytic
$2$-isometry such that $\dim \ker M_{z,\mu}^*=1$ (see
\cite[Corollary~ 3.8(a)]{R-1}), we infer from Theorem~
\ref{model}(a) that $M_{z,\mu}$ is unitarily
equivalent to a $2$-isometric unilateral weighted
shift $S$ with weights
$\{\xi_n(\lambda)\}_{n=0}^{\infty}$, where $\lambda
\in [1,\infty)$ (note that the closed support of $E$
is a one-point set). In view of the previous
paragraph, $S$ is unitarily equivalent to the operator
$M_{z,\alpha m}$ of multiplication by the coordinate
function $z$ on $\mathcal D(\alpha \, m)$, where
$\alpha=\lambda^2-1\in \rbb_+$. Applying
\cite[Theorem~ 5.2]{R-1}, we conclude that $\mu=\alpha
\, m$.
   \end{proof}
   \begin{remark} \label{Dirsh}
It follows from the first paragraph of the proof of
Proposition~ \ref{2isokk} that $\mathcal D(m)$ is the
Dirichlet space and the operator of multiplication by
the coordinate function $z$ on $\mathcal D(m)$ ({\em
Dirichlet shift}\/) is unitarily equivalent to the
unilateral weighted shift with weights
$\left\{\sqrt{\frac{n+2}{n+1}}\/\right\}_{n=0}^{\infty}.$
This means that Dirichlet shift is the fundamental
example of a cyclic analytic $2$-isometry which
satisfies the kernel condition.
   \hfill $\diamondsuit$
   \end{remark}
   \section{\label{Sec6}The Cauchy dual subnormality problem via the kernel condition}
In this section, we answer the Cauchy dual
subnormality problem in the affirmative for
$2$-isometries that satisfy the kernel condition (see
Theorem~ \ref{cdsubn}). We provide two proofs, the
first of which depends on the model Theorem~
\ref{model}, while the second does not.

Before doing this, we recall some definitions and
state two useful facts related to classical moment
problems. A sequence
$\gammab=\{\gamma_n\}_{n=0}^{\infty}\subseteq \rbb$ is
said to be a {\em Hamburger} (resp., {\em Stieltjes},
{\em Hausdorff}\/) {\em moment sequence} if there
exists a positive Borel measure $\mu$ on $\rbb$
(resp., $\rbb_+$, $[0,1]$) such that $\gamma_n = \int
t^n d \mu(t)$ for every $n\in \zbb_+$; such a $\mu$ is
called a {\em representing measure} of $\gammab$. Note
that a Hausdorff moment sequence is always determinate
as a Hamburger moment sequence, i.e., it has a unique
representing measure on $\rbb$. We refer the reader to
\cite{B-C-R,sim} for more information on moment
problems. The following lemma describes representing
measures of special rational-type Hausdorff moment
sequences.
   \begin{lemma}  \label{over}
Let $a,b\in \rbb$ be such that $a + bn \neq 0$ for
every $n\in\zbb_+$ and let $\gammab_{a,b} =
\{\gamma_{a,b}(n)\}_{n=0}^{\infty}$ be a sequence
given by $\gamma_{a,b}(n) = \frac{1}{a+bn}$ for all
$n\in\zbb_+$. Then $\gammab_{a,b}$ is a Hamburger
moment sequence if and only if $a>0$ and $b \Ge 0$. If
this is the case, then $\gammab_{a,b}$ is a Hausdorff
moment sequence and its unique representing measure
$\mu_{a,b}$ is given by
   \begin{align*}
\mu_{a,b} (\varDelta) =
   \begin{cases}
   \frac{1}{b} \int_{\varDelta} t^{\frac{a}{b}-1} d t
   & \text{ if } a > 0 \text{ and } b > 0,
   \\[1ex]
   \frac 1a \delta_1(\varDelta) & \text{ if } a > 0
   \text{ and } b=0,
   \end{cases}
\quad \varDelta \in \borel{[0,1]}.
   \end{align*}
   \end{lemma}
   \begin{proof}
If $\gammab_{a,b}$ is a Hamburger moment sequence,
then $\gamma_{a,b}(2n) > 0$ for all $n\in \zbb_+,$
which implies that $a > 0$ and $b\Ge 0$. Conversely,
if $a>0$ and $b\Ge 0,$ then applying the well-known
integral formula
   \begin{align} \label{log}
\int_{[0,1]} t^{\alpha} d t =
   \begin{cases}
\frac{1}{\alpha + 1} & \text{if } \alpha \in
(-1,\infty),
   \\[1ex]
\infty & \text{if } \alpha \in (-\infty, -1],
   \end{cases}
   \end{align}
one can easily verify that $\gammab_{a,b}$ is a
Hausdorff moment sequence with a representing measure
$\mu_{a,b}$.
   \end{proof}
The next property of moment sequences, whose prototype
appeared in \cite[Note on p.\ 780]{c-s-s}, can be
deduced from the Hamburger theorem \cite[Theorem~
6.2.2]{B-C-R}, the Stieltjes theorem \cite[Theorem~
6.2.5]{B-C-R} and the Hausdorff theorem \cite[Theorem~
4.6.11]{B-C-R}.
   \begin{lemma} \label{calka}
Let $(X,\mathcal A,\mu)$ be a measure space and
$\{\gamma_n\}_{n=0}^{\infty}$ be a sequence of
$\mathcal A$-measurable real valued functions on $X$.
Assume that $\{\gamma_n(x)\}_{n=0}^{\infty}$ is a
Hamburger $($resp., Stieltjes, Hausdorff\/$)$ moment
sequence for $\mu$-almost every $x\in X$ and $\int_X
|\gamma_n| d \mu < \infty$ for all $n\in \zbb_+$. Then
$\{\int_X \gamma_n d \mu\}_{n=0}^{\infty}$ is a
Hamburger $($resp., Stieltjes, Hausdorff\/$)$ moment
sequence.
   \end{lemma}
   We are now in a position to prove the main result
of this section.
   \begin{theorem} \label{cdsubn}
Let $T\in \B(\hh)$ be a $2$-isometry such that $T^*T
(\ker T^*) \subseteq \ker T^*.$ Then $T'$ is a
subnormal contraction such that $T'^*T' (\ker T'^*)
\subseteq \ker T'^*$ and
   \begin{align} \label{cdformula}
T^{\prime *n}T^{\prime n} = (n
T^*T-(n-1)I)^{-1}=(T^{*n}T^n)^{-1}, \quad n \in
\zbb_+.
   \end{align}
   \end{theorem}
   \begin{proof}[Proof I]
Applying Proposition~ \ref{kc-Cu}, we get $T'^*T'
(\ker T'^*) \subseteq \ker T'^*.$ By Proposition~
\ref{restrcd} and Theorem~ \ref{model}(a), it suffices
to consider the case of $T=W$, where $W$ is an
operator valued unilateral weighted shift on
$\ell^2_{\mathcal M}$ with weights
$\{W_n\}_{n=0}^{\infty}$ given by \eqref{wagi}. Since
the weights of $W$ are invertible, selfadjoint and
commuting, we infer from \eqref{aopws2} that
   \begin{align} \label{ajjaj}
W^{\prime}(h_0, h_1, \ldots) = (0, (W_0)^{-1} h_0,
(W_1)^{-1} h_1, \ldots), \quad (h_0, h_1, \ldots) \in
\ell^2_{\mathcal M},
   \end{align}
   which means that $W^{\prime}$ is an operator valued
unilateral weighted shift on $\ell^2_{\mathcal M}$
with weights $\{(W_n)^{-1}\}_{n=0}^{\infty}$. Thus, by
the commutativity of $\{W_n\}_{n=0}^{\infty}$ and the
inversion formula for spectral integral, we have (see
\eqref{wmndef})
   \begin{align*}
(W^{\prime})_{n,0} =
(W_{n,0})^{-1}\overset{\eqref{wmn}}=\int_{[1,\eta]}
\frac{1}{\sqrt{1+n(x^2-1)}} \, E(d x), \quad n \in
\zbb_+,
   \end{align*}
where $\eta:=\sup(\supp{E})$. This implies that
   \begin{align*}
\|(W^{\prime})_{n,0}f\|^2 = \int_{[1,\eta]}
\frac{1}{1+n(x^2-1)} \, \langle E(d x)f, f \rangle
\quad n \in \zbb_+, \, f \in \mathcal M.
   \end{align*}
Using Lemmata~ \ref{over} and \ref{calka}, we deduce
that $\{\|(W^{\prime})_{n,0}f\|^2\}_{n=0}^{\infty}$ is
a Stieltjes moment sequence for every $f\in \mathcal
M$. Therefore, by \cite[Corollary~ 3.3]{St} (cf.\
\cite[Theorem~ 3.2]{lam}), $W^{\prime}$ is a subnormal
operator which, by \eqref{2hypcon}, is a contraction.

It remains to prove \eqref{cdformula}. Using
\eqref{opvuws} and \eqref{aopws} as well as the fact
that the weights of $W$ are selfadjoint, invertible
and commuting, we get
   \begin{gather}  \notag
W^{*n}W^n = \bigoplus_{j=0}^{\infty}
(W_{n+j,j})^*W_{n+j,j} = \bigoplus_{j=0}^{\infty}
(W_{n+j,j})^2, \quad n \in \zbb_+,
   \\ \label{wprim}
W^{\prime *n} W^{\prime n} = \bigoplus_{j=0}^{\infty}
((W^{\prime})_{n+j,j})^2 = \bigoplus_{j=0}^{\infty}
(W_{n+j,j})^{-2}, \quad n \in \zbb_+.
   \end{gather}
This implies that $W^{\prime *n} W^{\prime n}
=(W^{*n}W^n)^{-1}$ for all integers $n\Ge 0$. Since
   \begin{align*}
(W_{n+j,j})^{-2} \overset{\eqref{wmn}} =
\int_{[1,\eta]} \frac{1+j(x^2-1)}{1+(n+j)(x^2-1)} \,
E(d x) \overset{\eqref{wagi}} = (n W_j^*W_j - (n-1)
I)^{-1}
   \end{align*}
for all $j, n \in \zbb_+$, we infer from
\eqref{aopws2} and \eqref{wprim} that
   \begin{align*}
W^{\prime *n} W^{\prime n} = (n W^*W-(n-1)I)^{-1},
\quad n \in \zbb_+,
   \end{align*}
which completes Proof I of Theorem~ \ref{cdsubn}.
   \end{proof}
   \begin{proof}[Proof II]
By Proposition~ \ref{kc-Cu}, $T'^*T' (\ker T'^*)
\subseteq \ker T'^*.$ It follows from \eqref{tpr3} and
Lemma~ \ref{char-1}(iii) that
   \begin{align*}
(T'^*T')^{-1}T' - 2T' + T' T'^*T' & = T^*T T' - 2T' +
T' (T^*T)^{-1}
   \\
   &= (T^*T^2 - 2T + T') (T^*T)^{-1} = 0,
   \end{align*}
which implies that
   \begin{align} \label{eqn}
(T'^*T')^{-1}T' = 2T' - T' T'^*T'.
   \end{align}
Since $T^*T \Ge I$ (see \cite[Lemma 1]{R-0}), it
follows that $\sigma(T^*T)\subseteq [1, \|T\|^2].$
Applying \eqref{tpr3}, we get
   \begin{align} \notag
n(T'^*T')^{-1}-(n-1)I &= nT^*T-(n-1)I
   \\  \label{ciach}
&=n\left(T^*T-\frac{n-1}{n}I\right), \quad n\in \nbb,
   \end{align}
which implies that the operator
$n(T'^*T')^{-1}-(n-1)I$ is invertible in $\B(\hh)$ for
every $n \in \nbb$.

Now we prove the first equality in \eqref{cdformula}
by induction on $n$. The cases $n=0,1$ are obvious.
Assume that this equality holds for a fixed $n \in
\nbb$. Then, by the induction hypothesis and
\eqref{ciach}, we have
   \begin{align*}
T^{\prime *n}T^{\prime n} (n(T'^*T')^{-1}-(n-1)I)=I,
\quad n\in \nbb.
   \end{align*}
Multiplying by $T'^*$ and $T'$ from the left-hand side
and the right-hand side, respectively, both sides of
the above equality, we get
   \begin{align} \notag
T'^*T'&=T^{\prime *(n+1)}T^{\prime n}
(n(T'^*T')^{-1}-(n-1)I)T'
   \\      \notag
&\hspace{-.8ex}\overset{\eqref{eqn}}= T^{\prime
*(n+1)}T^{\prime n}((n+1)T' - nT' T'^*T')
   \\            \label{ciach1}
&= T^{\prime *(n+1)}T^{\prime (n+1)}((n+1)I -
nT'^*T'), \quad n\in \nbb.
   \end{align}
Noting that
   \begin{align} \label{tpr5}
((n+1)I - nT'^*T') (T'^*T')^{-1}
\overset{\eqref{tpr3}}= (n+1)T^*T-nI, \quad n \in
\nbb,
   \end{align}
we infer from \eqref{ciach} that the operator $(n+1)I
- nT'^*T'$ is invertible in $\B(\hh)$ for every $n \in
\nbb$. This combined with \eqref{ciach1} yields
   \begin{align*}
T^{\prime *(n+1)}T^{\prime (n+1)} = T'^*T'((n+1)I -
nT'^*T')^{-1} \overset{\eqref{tpr5}}=
((n+1)T^*T-nI)^{-1}.
   \end{align*}
This completes the induction argument.

Since $T$ is a $2$-isometry, we deduce by using
induction on $n$ that (see also \cite[Proposition~
4.5]{Ja})
   \begin{align*}
T^{*n}T^n = nT^*T-(n-1)I, \quad n \in \zbb_+.
   \end{align*}
This combined with the first equality in
\eqref{cdformula} gives the second one in
\eqref{cdformula}.

It follows from the first equality in
\eqref{cdformula} and the Stone-von Neumann calculus
for selfadjoint operators that
   \begin{align} \label{tprhaus}
\|T'^n f\|^2 = \int_{[1, \|T\|^2]} \frac{1}{1 + n
(x-1)} \langle G(d x)f, f \rangle, \quad n\in \zbb_+,
\, f \in \hh,
   \end{align}
where $G$ is the spectral measure of $T^*T$. This
together with Lemmata~ \ref{over} and \ref{calka}
implies that $\{\|T'^n f\|^2\}$ is a Stieltjes moment
sequence for every $f\in \hh$. By Lambert's theorem
(see Theorem~ \ref{Lam}), $T'$ is a subnormal operator
which, by \eqref{2hypcon}, is a contraction. This
completes the proof.
   \end{proof}
   \begin{remark}
Regarding Theorem~ \ref{cdsubn}, it is worth
mentioning that due to \eqref{tprhaus}, for every
$f\in \hh,$ the Hausdorff moment sequence $\{\|T'^n
f\|^2\}_{n=0}^{\infty}$ comes from the Hausdorff
moment sequences $\{\gammab_{1,x-1} \colon x\in
[1,\infty)\}$ appearing in Lemma~ \ref{over} via the
integration procedure described in Lemma~ \ref{calka}.
   \hfill $\diamondsuit$
   \end{remark}
We now state a few corollaries to Theorem~
\ref{cdsubn}.
   \begin{corollary}
If $T\in \B(\hh)$ is a $2$-isometry satisfying the
kernel condition, then the family $\{T^{*n}T^n\colon
n\in \zbb_+\} \cup \{T^{\prime *n}T^{\prime n}\colon
n\in \zbb_+\}$ consists of commuting selfadjoint
operators.
   \end{corollary}
   \begin{corollary} \label{nthpow}
Suppose $T\in \B(\hh)$ is a $2$-isometry satisfying
the kernel condition. Then for every $n\in \nbb,$
$T^n$ is a $2$-isometry satisfying the kernel
condition and $(T^n)'$ is a subnormal contraction
satisfying the kernel condition. In particular, this
is the case for $2$-isometric unilateral weighted
shifts.
   \end{corollary}
   \begin{proof}
Using the fact that positive integral powers of
$2$-isometries are $2$-isometries (see \cite[Theorem~
2.3]{Ja}), Theorem \ref{cdsubn} and Proposition
\ref{kc-Cu}, it suffices to prove that $T^n$ satisfies
the kernel condition. In view of Theorem~
\ref{model}(a), we may assume without loss of
generality that $T=W$, where $W$ is as in this
theorem. Then $\ker W^*=\mathcal M \oplus \{0\} \oplus
\{0\} \oplus \ldots.$ Using induction and the formulas
\eqref{opvuws} and \eqref{aopws}, we deduce that for
every $n\in \nbb$, $W^n$ satisfies the kernel
condition.
   \end{proof}
The next corollary is of some importance because the
single equality of the form $p(T,T^*)=0,$ where $p$ is
a polynomial in two non-commuting variables of degree
$5$, yields subnormality of $T.$ The reader is
referred to \cite[Theorem~ 5.4 and Proposition~
7.3]{St0} for an example of an unbounded non-subnormal
formally normal operator annihilated by a polynomial
$p(z,\bar z)$ of (the lowest possible) degree $3$.
   \begin{corollary} \label{special}
Suppose $T\in \B(\hh).$ Then
   \begin{enumerate}
   \item[(i)] $T'$ is a subnormal contraction if
$T$ is left-invertible and
   \begin{align} \label{eqeq1}
(T^*T^2 T^*-2TT^* + I)T=0,
   \end{align}
   \item[(ii)] $T$ is a subnormal contraction if
$T$ is left-invertible and
   \begin{align} \label{eqeq2}
(T^*T^2 T^*-2T^*T + I)T=0.
   \end{align}
   \end{enumerate}
Moreover, in both cases $T$ and $T'$ satisfy the
kernel condition.
   \end{corollary}
   \begin{proof}
(i) Combining Proposition~ \ref{kc-Cu} and Lemma~
\ref{char-1} with Theorem~ \ref{cdsubn} yields (i) and
shows that $T$ and $T'$ satisfy the kernel condition.

(ii) Apply (i) to $T'$ in place of $T$ and use
\eqref{tpr0}.
   \end{proof}
   \begin{remark} \label{irownii}
A careful look at the proof of Corollary~
\ref{special} reveals that the assertions (i) and (ii)
are equivalent. In fact, a left-invertible operator
$T\in \B(\hh)$ satisfies \eqref{eqeq1} (resp.\
\eqref{eqeq2}) if and only if $T$ (resp.\ $T'$) is a
$2$-isometry which satisfies the kernel condition.
   \hfill $\diamondsuit$
   \end{remark}
The following example shows that some classical
operators on Hilbert spaces of analytic functions are
closely related to Corollary~ \ref{special}.
   \begin{example}
For $l \in \nbb,$ consider the reproducing kernel
   \begin{align*}
\kappa_{l}(z, w) = \frac{1}{(1-z\overline{w})^{l}},
\quad z,w \in \mathbb D,
   \end{align*}
where $\mathbb D = \{z\in \C \colon |z|< 1\}$ and
$\overline{w}$ stands for the complex conjugate of
$w.$ Let $\mathscr H_{l}$ denote reproducing kernel
Hilbert space associated with $\kappa_{l}$ and let
$M_{z,l}$ be the operator of multiplication by the
coordinate function $z$ on $\mathscr H_{l}$. It is
well-known that $M_{z, l}$ is a subnormal contraction
for every $l \in \nbb$ (see \cite[Section~ 1]{Ag}).
Note that $\mathscr H_{1}$ is the Hardy space $H^2$
and the operator $M_{z, 1}$ ({\em Szeg\"{o} shift}) is
an isometry, i.e., $M^*_{z, 1}M_{z, 1}=I.$ In turn,
$\mathscr H_{2}$ is the Bergman space. Using the
standard orthonormal basis of $\mathscr H_{2},$ we
deduce that the operator $M_{z, 2}$ ({\em Bergman
shift}) is unitarily equivalent to the unilateral
weighted shift with weights
$\Big\{\sqrt{\frac{n+1}{n+2}}\/\Big\}_{n=0}^{\infty}.$
It is now easily seen that $M_{z, 2}$ is
left-invertible and satisfies the following identity
   \begin{align*}
M^*_{z, 2}M^2_{z, 2}M^*_{z, 2} -2M^*_{z, 2}M_{z, 2} +
I =0.
   \end{align*}
This together with Remark~ \ref{irownii} implies that
the Cauchy dual $M'_{z, 2}$ of $M_{z, 2}$ is a
$2$-isometry satisfying the kernel condition (cf.\
Remark~ \ref{Dirsh}).
   \hfill $\diamondsuit$
   \end{example}
Below we show that Theorem~ \ref{cdsubn} is no longer
true for $3$-isometries. Since $m$-isometries are
$(m+1)$-isometries (see \cite[p.\ 389]{Ag-St}), we see
that Theorem~ \ref{cdsubn} is not true for
$m$-isometries with $m\Ge 3.$
   \begin{example} \label{treiso}
Let $T$ be the unilateral weighted shift in
$\ell^2(\zbb_+)$ with weights
$\Big\{\sqrt{\frac{\phi(n+1)}{\phi(n)}}\/\Big\}_{n=0}^{\infty}$,
where $\phi(n)=n^2+1$ for $n\in \zbb_+$. It is a
matter of routine to verify that $T$ is a $3$-isometry
(one can also use \cite[Theorem~ 1]{Ab-Le}). Clearly,
$T$ is left-invertible and satisfies the kernel
condition. Since the Cauchy dual $T'$ of $T$ is the
unilateral weighted shift with weights $\Big\{\sqrt{
\frac{\phi(n)}{\phi(n+1)}}\/ \Big\}_{n=0}^{\infty}$
(see \eqref{ajjaj}), we verify easily that $T'$ is a
contraction and
   \begin{align*}
\sum_{n=0}^4 (-1)^n {4 \choose n}\|T'^n e_0\|^2 =
\sum_{n=0}^4 (-1)^n {4 \choose n} \frac{1}{\phi(n)} <
0.
   \end{align*}
In view of Theorem~ \ref{Ag-Emb}, the Cauchy dual
operator $T'$ is not subnormal.
   \hfill $\diamondsuit$
   \end{example}
   \section{\label{Sec7}The Cauchy dual subnormality problem for
quasi-Brownian isometries} This section deals with a
class of $2$-isometries which we propose to call
quasi-Brownian isometries. Our goal here is to solve
the Cauchy dual subnormality problem affirmatively
within this class (see Theorem~ \ref{BrownianG}). It
turns out that quasi-Brownian isometries do not
satisfy the kernel condition unless they are
isometries (see Corollary~ \ref{binkc}). In Section~
\ref{Sec11}, we exhibit an example of a $2$-isometry
$T \in \B(\hh)$ such that $T$ does not satisfy the
kernel condition, $T$ is not a quasi-Brownian isometry
and the Cauchy dual operator $T'$ is a subnormal
contraction (see Example~ \ref{nbnkcsub}).

We say that an operator $T\in \B(\hh)$ is a {\em
quasi-Brownian isometry} if $T$ is a $2$-isometry such
that
   \begin{align} \label{qbdef}
\text{$\triangle_T T = \triangle_T^{1/2} T
\triangle_T^{1/2},$ where $\triangle_T :=T^*T-I.$}
   \end{align}
(Recall that $\triangle_T \Ge 0$ for any $2$-isometry
$T.$) In \cite{Maj} such operators are called {\em
$\triangle_T$-regular} $2$-isometries. As we see below
(Corollary~ \ref{brareqbr} and Example~ \ref{mewa}),
the notion of a quasi-Brownian isometry generalizes
that of a Brownian isometry introduced by Agler and
Stankus in \cite{Ag-St}; the latter notion arose in
the study of the time shift operator on a modified
Brownian motion process. Here we do not include the
rather technical definition of a Brownian isometry as
we do not need it. Instead, we define a Brownian
isometry by using \cite[Theorem~ 5.48]{Ag-St}. Namely,
an operator $T\in \B(\hh)$ is said to be a {\it
Brownian isometry} if $T$ is a $2$-isometry such that
   \begin{align} \label{defbi}
\triangle_T \triangle_{T^*} \triangle_T = 0.
   \end{align}

Before proving the main result of this section, we
state slightly improved versions of \cite[Proposition~
5.1]{Maj} and \cite[Proposition~ 5.37 and Theorem~
5.48]{Ag-St}.
   \begin{theorem} \label{Maj-Mb-Su}
If $T\in \B(\hh),$ then the following conditions are
equivalent{\em :}
   \begin{enumerate}
   \item[(i)] $T$ is a quasi-Brownian isometry $($resp.,
Brownian isometry$)$,
   \item[(ii)] $T$ has the block matrix form
   \begin{align} \label{brep}
T = \left[\begin{array}{cc} {V } & {E}\\ {0} & {U}
\end{array}\right]
   \end{align}
with respect to an orthogonal decomposition $\hh=\hh_1
\oplus \hh_2$ $($one of the summands may be absent$)$,
where $V\in \B(\hh_1),$ $E\in \B(\hh_2,\hh_1)$ and
$U\in \B(\hh_2)$ are such that
   \begin{gather}  \label{vsvi2}
\text{$V^*V=I$, $V^*E=0$, $U^*U=I$ and $UE^*E=E^*EU$}
   \\ \label{vsvi} \text{$($resp., $V^*V=I$, $V^*E=0$, $U^*U=I=UU^*$
and $UE^*E=E^*EU$$),$}
   \end{gather}
   \item[(iii)] $T$ is either isometric or it
has the block matrix form \eqref{brep} with respect to
a nontrivial orthogonal decomposition $\hh=\hh_1
\oplus \hh_2,$ where $V\in \B(\hh_1),$ $E\in
\B(\hh_2,\hh_1)$ and $U\in \B(\hh_2)$ satisfy
\eqref{vsvi2} $($resp., \eqref{vsvi}$)$ and
\mbox{$\ker E = \{0\}$}.
   \end{enumerate}
   \end{theorem}
   \begin{proof}
That (i) implies (iii) follows from \cite[Proposition~
5.1]{Maj} (resp., the proof of \cite[Theorem~
5.48]{Ag-St}). Obviously, (iii) implies (ii). Finally,
it is a matter of routine to show that (ii) implies
(i).
   \end{proof}
   \begin{corollary} \label{brareqbr}
Every Brownian isometry is a quasi-Brownian isometry.
   \end{corollary}
   \begin{remark}
Note that if $T\in \B(\hh)$ has the block matrix form
\eqref{brep}, where $V\in \B(\hh_1),$ $E\in
\B(\hh_2,\hh_1)$ and $U\in \B(\hh_2)$ satisfy
\eqref{vsvi2}, then
   \begin{align*}
\|T^*T-I\|=\|E\|^2.
   \end{align*}
This means that the norm of the operator $E$ appearing
in \cite[Proposition~ 5.37]{Ag-St} and
\cite[Proposition~ 5.1]{Maj} must equal $1.$
   \hfill $\diamondsuit$
   \end{remark}
The converse to Corollary~ \ref{brareqbr} is not true.
   \begin{example}  \label{mewa}
Let $V\in \B(\hh_1),$ $E\in \B(\hh_2,\hh_1)$ and $U\in
\B(\hh_2)$ be isometric operators such that $U$ is not
unitary and $V^*E=0$ (which is always possible). By
Theorem~ \ref{Maj-Mb-Su}, we see that the
corresponding operator $T$ given by \eqref{brep} is a
quasi-Brownian isometry. However, $T$ is not a
Brownian isometry because
   \begin{align*}
\triangle_T \triangle_{T^*} \triangle_T =
\left[\begin{array}{cc} 0 & 0 \\ {0} & UU^* - I
\end{array}\right],
   \end{align*}
which, by the choice of $U$, implies that $\triangle_T
\triangle_{T^*} \triangle_T \neq 0$.
   \hfill $\diamondsuit$
   \end{example}
Now we are in a position to show that the Cauchy dual
operator of a quasi-Brownian isometry is a subnormal
contraction.
   \begin{theorem} \label{BrownianG}
   Suppose $T\in \B(\hh)$ is a quasi-Brownian
isometry. Then $T'$ is a subnormal contraction such
that
   \begin{align} \label{formula-2}
T'^{*n}T'^n = (I+T^*T)^{-1}(I+(T^*T)^{1- 2n}), \quad n
\in \zbb_+.
   \end{align}
   \end{theorem}
   \begin{proof}
It follows from Theorem~ \ref{Maj-Mb-Su} that $T$ has
the block matrix form \eqref{brep} with respect to an
orthogonal decomposition $\hh=\hh_1 \oplus \hh_2$,
where $V\in \B(\hh_1),$ $E\in \B(\hh_2,\hh_1)$ and
$U\in \B(\hh_2)$ satisfy \eqref{vsvi2}. Without loss
of generality, we may assume that $T$ is not an
isometry, which implies that $E\neq 0$. By
\eqref{vsvi2}, we have
   \begin{align} \label{ipee}
T^*T=I \oplus Q, \text{ where } Q=I+E^*E.
   \end{align}
Clearly, $Q$ is selfadjoint and invertible in
$\B(\hh).$ For $n \in \zbb_+,$ we define the rational
function $r_n\colon [1,\infty) \to (0,\infty)$ by
   \begin{align*}
r_n(x) = \frac{1+{x^{1-2n}}}{1+x} , \quad x \in [1,
\infty).
   \end{align*}
Since, by \eqref{ipee} (or by \cite[Lemma 1]{R-0}),
$T^*T \Ge I$ we see that $\sigma(T^*T)\subseteq [1,
\|T\|^2].$ Applying the functional calculus (see
\cite[Theorem~ VIII.2.6]{Con}), we deduce that
\eqref{formula-2} is equivalent~ to
   \begin{align}  \label{formula-A}
T'^{*n}T'^n = r_n(T^*T), \quad n \in \zbb_+.
   \end{align}
We prove \eqref{formula-A} by induction on $n.$ The
case $n=0$ is obviously true. Suppose that
\eqref{formula-A} holds for some unspecified $n\in
\zbb_+.$ Using \eqref{ipee} and functional calculus,
we get
   \begin{align} \label{rational}
r_k(T^*T) = I \oplus r_k(Q), \quad k\in \zbb_+.
   \end{align}
It is a matter of routine to verify that
   \begin{align} \label{rational2}
T' = \left[\begin{array}{cc} {V } & {EQ^{-1}}\\
{0} & {UQ^{-1}}
\end{array}\right].
   \end{align}
The induction hypothesis, the equalities $V^*E=0$ and
$UQ=QU$ and the functional calculus yield
   \begin{align*}
   T'^{*(n+1)}T'^{(n+1)} &\overset{\eqref{rational}}=
   T'^*(I \oplus r_n(Q))T'
   \\
&\overset{\eqref{rational2}}= \left[\begin{array}{cc} {V^* } & {0}\\
{Q^{-1}E^*} & {Q^{-1}U^*}
\end{array}\right]\left[\begin{array}{cc} {I } & {0}
   \\ {0} & {r_n(Q)}
\end{array}\right] \left[\begin{array}{cc} {V } & {EQ^{-1}}
   \\ {0} & {UQ^{-1}} \end{array}\right]
   \\[.2ex]
&\hspace{.7ex}= I \oplus (Q^{-1}-Q^{-2} +
Q^{-2}r_n(Q))
   \\
&\hspace{.7ex}= I \oplus r_{n+1}(Q)
   \\
&\overset{\eqref{rational}}= r_{n+1}(T^*T),
   \end{align*}
which completes the induction argument.

By \eqref{2hypcon}, $T'$ is a contraction. It follows
from \eqref{formula-A} and the Stone-von Neumann
calculus for selfadjoint operators that
   \begin{align} \label{tprhaus-qB}
\|T'^n f\|^2 = \int_{[1,\|T\|^2]} \left(\frac{1}{1+x}
+ \frac{x}{1+x} (x^{-2})^n\right)\langle G(d x)f, f
\rangle, \quad n\in \zbb_+, \, f \in \hh,
   \end{align}
where $G$ is the spectral measure of $T^*T.$ Now
applying Lemma~ \ref{calka}, we see that $\{\|T'^n
f\|^2\}_{n=0}^{\infty}$ is a Stieltjes moment sequence
for every $f\in \hh$. This together with Lambert's
theorem (see Theorem~ \ref{Lam}) completes the proof.
   \end{proof}
The techniques developed in proofs of Theorems~
\ref{cdsubn} and \ref{BrownianG} give the following.
   \begin{corollary}  \label{binkc}
Suppose $T\in \B(\hh)$ is a quasi-Brownian isometry
that satisfies the kernel condition. Then $T$ is an
isometry.
   \end{corollary}
   \begin{proof} Let $G$ be the spectral measure of
$T^*T$. Note that $0 \Le \frac{1}{1 + n (x-1)} \Le 1$
for all $x\in [1,\infty)$ and $n\in \zbb_+$, and
$\lim_{n\to \infty} \frac{1}{1 + n (x-1)} =
\chi_{\{1\}}(x)$ for all $x \in [1,\infty)$. Applying
Lebesgue's dominated convergence theorem to
\eqref{tprhaus}, we deduce that the sequence
$\{T^{\prime *n}T^{\prime n}\}_{n=1}^{\infty}$ of
positive operators converges to $G(\{1\})$ in the weak
and consequently in the strong operator topology. A
similar argument applied to \eqref{tprhaus-qB} shows
that $\{T^{\prime *n}T^{\prime n}\}_{n=1}^{\infty}$
converges to $\frac{1}{2}G(\{1\}) + (I + T^*T)^{-1}$
in the strong operator topology. Hence $G(\{1\})=I$
and thus $T$ is an isometry.
   \end{proof}
The so-called Brownian shifts introduced in
\cite[Definition 5.5]{Ag-St} are examples of Brownian
isometries which are not isometric, and thus by
Corollary~ \ref{binkc} they do not satisfy the kernel
condition. In turn, using Theorem~ \ref{Maj-Mb-Su},
one can show that the composition operator $C_{\phi}$
that appeared in \cite[Example~ 4.4]{Ja-1} (in
connection with the study of $2$-hyperexpansive
operators) with constant parameter sequence
$\{a_n\}_{n=-\infty}^{\infty}$ is a non-isometric
Brownian isometry which is not unitarily equivalent to
a Brownian shift.
   \section{\label{Sec8}$2$-isometric weighted shifts
on directed trees} Here we focus our attention on
$2$-isometric weighted shifts on directed trees. We
refer the reader to \cite[Chapters 2 and 3]{JJS} for
all definitions pertaining to directed trees and
weighted shifts on directed trees.

Let $\tcal = (V,E)$ be a directed tree (if not stated
otherwise, $V$ and $E$ stand for the sets of vertices
and edges of $\tcal$ respectively). If $\tcal$ has a
root, we denote it by $\rot$.
 We write $V^{\circ}=V\setminus \{\rot\}$ if $\tcal$
is rooted and $V^{\circ}=V$ otherwise. Given
$W\subseteq V$ and $n\in \zbb_+,$ we set
$\childn{n}{W}=W$ if $n=0$ and
$\childn{n}{W}=\child{\childn{n-1}{W}}$ if $n\Ge 1$,
where $\child{W} = \bigcup_{u\in W} \{v\in V \colon
(u,v) \in E\}$. We also set $\des{W} =
\bigcup_{n=0}^{\infty} \childn{n}{W}$. For brevity, we
write $\child{v}=\child{\{v\}}$,
$\childn{n}{v}=\childn{n}{\{v\}}$ and
$\des{v}=\des{\{v\}}$ whenever $v\in V$ and $n\in
\zbb_+$. A member of $\child{v}$ (resp., $\des{v}$) is
called a {\em child} (resp., a {\em descendant}) of
$v$. For $v\in V^{\circ}$, a unique $u\in V$ such that
$(u,v)\in E$ is called the {\em parent} of $v$ and
denoted by $\parent{v}$. We put $V^{\prime} = \{u \in
V\colon \child{u} \neq \emptyset\}$. If
$V=V^{\prime}$, we say that $\tcal$ is {\em leafless}.
By the {\em degree} of a vertex $v\in V$, in notation
$\deg{v}$, we understand the cardinality of
$\child{v}$. A directed tree whose each vertex is of
finite degree is called {\em locally finite}. If
$\tcal$ is rooted, then (see \cite[Corollary~
2.1.5]{JJS})
   \begin{align} \label{ind1}
V = \des{\rot} = \bigsqcup_{n=0}^\infty \childn{n}
{\rot} \quad \text{(the disjoint sum)}.
   \end{align}

   Below we give examples of directed trees playing an
essential role in this~ paper.
   \begin{example} \label{trzykrow}
(a) We begin with two classical directed trees, namely
   \begin{align*}
\text{$(\zbb_+,\{(n,n+1)\colon n \in \zbb_+\})$ and
$(\zbb,\{(n,n+1)\colon n \in \zbb\}),$}
   \end{align*}
which will be denoted simply by $\zbb_+$ and $\zbb,$
respectively. The directed tree $\zbb_+$ is rooted,
$\zbb$ is rootless and both are leafless.

(b) Following \cite[page 67]{JJS}, we define the
directed tree $\tcal_{\eta,\kappa} = (V_{\eta,\kappa},
E_{\eta,\kappa})$ by \allowdisplaybreaks
   \begin{align*}
   \begin{aligned} V_{\eta,\kappa} & = \big\{-k\colon k\in
J_\kappa\big\} \cup \{0\} \cup \big\{(i,j)\colon i\in
J_\eta,\, j \in J_{\infty}\big\},
   \\
E_{\eta,\kappa} & = E_\kappa \cup
\big\{(0,(i,1))\colon i \in J_\eta\big\} \cup
\big\{((i,j),(i,j+1))\colon i\in J_\eta,\, j\in
J_{\infty}\big\},
   \\
E_\kappa & = \big\{(-k,-k+1) \colon k\in
J_\kappa\big\},
   \end{aligned}
   \end{align*}
where $\eta \in \{2,3,4,\ldots\} \cup \{\infty\}$,
$\kappa \in \zbb_+ \cup \{\infty\}$ and $J_\iota = \{k
\in \zbb\colon 1 \Le k\Le \iota\}$ for $\iota \in
\zbb_+ \sqcup \{\infty\}$. The directed tree
$\tcal_{\eta,\kappa}$ is leafless and $0$ is the only
vertex of $\tcal_{\eta,\kappa}$ of degree greater than
$1$. It is rooted if and only if $\kappa < \infty$,

(c) Let $l\in \{2,3,4, \ldots\}.$ We say that a
directed tree $\tcal$ is a {\em quasi-Brownian
directed tree of valency} $l$ (or simply a {\em
quasi-Brownian directed tree}) if
   \begin{align} \notag
&\bullet\text{there exists $u_0\in V$ such that
$\deg{u_0}=l,$}
   \\  \label{(B2)}
&\bullet\text{each vertex $u\in V$ is of degree $1$ or
$l,$}
   \\ \label{(B3)}
&\bullet\text{if $u\in V$ is such that $\deg{u}=1$ and
$v\in \child{u},$ then $\deg{v}=1,$}
   \\  \label{(B4)}
&\bullet\left\{\begin{minipage}{70ex} for every $u\in
V$ with $\deg{u}=l,$ there is exactly one $v\in
\child{u}$ such that $\deg{v}=l$ and the remaining
$l-1$ vertices in $\child{u}$ are of degree~ $1.$
   \end{minipage}\right.
   \end{align}
A quasi-Brownian directed tree of valency $l \Ge 3$
can be defined as follows:
   \begin{align*}
V & = X \times V_{l-1,0},
   \\
E &= \Big\{\big((n,0), (n+1,0)\big)\colon n\in X\Big\}
   \\
   &\hspace{10ex} \sqcup \bigsqcup_{n\in X}
\Big\{\big((n,u),(n,v)\big)\colon u,v \in V_{l-1,0},
\, (u,v) \in E_{l-1,0}\Big\},
   \end{align*}
where $X=\zbb_+$ in the rooted case and $X=\zbb$ in
the rootless case. Geometrically, it is obtained by
``gluing'' to each $n\in X$ a copy of the directed
tree $\tcal_{l-1,0}$ defined in (b). A similar
construction can be performed for $l=2$. Using
\cite[Proposition~ 2.1.4]{JJS} and \cite[Proposition~
2.2.1]{BJJS}, one can verify that there are only two
(up to graph isomorphism) quasi-Brownian directed
trees of valency $l,$ one with root, the other
without.
   \hfill $\diamondsuit$
   \end{example}
Applying induction on $n$, we see that if $\tcal$ is a
quasi-Brownian directed tree of valency $l$, then for
every $u\in V$ with $\deg{u}=l$ and for every $n\in
\zbb_+,$
   \begin{enumerate}
   \item[$\bullet$]
there is exactly one vertex $v\in \childn{n}{u}$ such
that $\deg{v}=l$ and the remaining vertices in
$\childn{n}{u}$ are of degree $1,$
   \item[$\bullet$] $\childn{n}{u}$ consists of $1 + n(l-1)$
vertices.
   \end{enumerate}
Obviously, quasi-Brownian directed trees are locally
finite and leafless.

The following lemma characterizes rooted
quasi-Brownian directed trees.
   \begin{lemma} \label{qbchar}
Let $\tcal$ be a rooted and leafless directed tree
such that $\deg{\omega} \in \{2,3,4, \ldots\}.$ Set
$l=\deg{\omega}$. Then the following conditions are
equivalent{\em :}
   \begin{enumerate}
   \item[(i)] $\tcal$ is a quasi-Brownian directed
tree of valency $l,$
   \item[(ii)] $\tcal$ satisfies \eqref{(B3)} and \eqref{(B4)},
   \item[(iii)] $\tcal$ satisfies \eqref{(B2)} and
the following equation
   \begin{align} \label{Hak3}
\sum_{v\in \child{u}} \deg{v} = 2 \deg{u}-1, \quad u
\in V,
   \end{align}
   \item[(iv)] $\tcal$ satisfies \eqref{Hak3} and
the following condition
   \begin{align} \label{Hak2}
\text{$\deg{u}=\deg{v}$ whenever $u\in V,$ $v\in
\child{u}$ and $\deg{v}\Ge 2.$}
   \end{align}
   \end{enumerate}
   \end{lemma}
   \begin{proof}
The implications (i)$\Rightarrow$(ii),
(iii)$\Rightarrow$(i) and (iii)$\Rightarrow$(iv) are
easily seen to be true. The implication
(ii)$\Rightarrow$(iii) follows from \eqref{ind1} by
induction.

   (iv)$\Rightarrow$(iii) Suppose $v \in V^{\circ}$ is
such that $\deg{v}\Ge 2.$ By \cite[Proposition~
2.2.1]{BJJS}, there exists $n\in \nbb$ such that
$\mathsf{par}^n(v)=\omega.$ It follows from
\eqref{Hak2} that $\deg{\parent{v}}=\deg{v}.$ By
induction, $\deg{v}=\deg{\omega}=l,$ which shows that
$\tcal$ satisfies \eqref{(B2)}.
   \end{proof}
Let $\tcal=(V,E)$ be a directed tree. In what follows
$\ell^2(V)$ stands for the Hilbert space of square
summable complex functions on $V$ equipped with the
standard inner product. If $W$ is a nonempty subset of
$V,$ then we regard the Hilbert space $\ell^2(W)$ as a
closed vector subspace of $\ell^2(V)$ by identifying
each $f\in \ell^2(W)$ with the function $\widetilde f
\in \ell^2(V)$ which extends $f$ and vanishes on the
set $V \setminus W$. Note that the set $\{e_u\}_{u\in
V}$, where $e_u\in \ell^2(V)$ is the characteristic
function of $\{u\}$, is an orthonormal basis of
$\ell^2(V)$. Given a system $\lambdab =
\{\lambda_v\}_{v\in V^{\circ}} \subseteq \C$, we
define the operator $\slam$ in $\ell^2(V)$, called a
{\em weighted shift on} $\tcal$ with weights
$\lambdab$ (or simply a weighted shift on $\tcal$), as
follows
   \begin{align*}
   \begin{aligned}
\mathscr D(\slam) & = \{f \in \ell^2(V) \colon
\varLambda_\tcal f \in \ell^2(V)\},
   \\
\slam f & = \varLambda_\tcal f, \quad f \in \mathscr
D(\slam),
   \end{aligned}
   \end{align*}
where $\mathscr D(\slam)$ stands for the {\em domain}
of $\slam$ and $\varLambda_\tcal\colon \C^{V} \to
\C^{V}$ is defined by
   \begin{align*}
(\varLambda_\tcal f) (v) =
   \begin{cases}
\lambda_v \cdot f\big(\parent v\big) & \text{if } v\in
V^\circ,
   \\
   0 & \text{if } v \text{ is a root of } \tcal,
   \end{cases}
   \qquad f \in \C^{V}.
   \end{align*}

Now we collect some properties of weighted shifts on
directed trees that are needed in this paper. We also
show that weighted shifts on rooted directed trees are
completely non-unitary. This is no longer true even
for isometric weighted shifts on rootless directed
trees.

From now on, we adopt the convention that $\sum_{v\in
\emptyset} x_v = 0$. Recall also that $\triangle_T
=T^*T-I$ whenever $T\in \B(\hh)$ (see \eqref{qbdef}).
   \begin{lemma} \label{basicws}
Let $\slam$ be a weighted shift on $\tcal$ with
weights $\lambdab=\{\lambda_v\}_{v\in V^{\circ}}$.
Then
   \begin{enumerate}
   \item[(i)] $e_u$ is in $\mathcal D(\slam)$ if and
only if $\sum_{v \in \child{u}}|\lambda_v|^2 <
\infty;$ if $e_u \in \mathscr D(\slam)$, then $\slam
e_u = \sum_{v \in \child{u}}\lambda_v e_v$ and
$\|\slam e_u\|^2 = \sum_{v \in
\child{u}}|\lambda_v|^2,$
   \item[(ii)]
$\slam \in \B(\ell^2(V))$ if and only if $\sup_{u\in
V} \sum_{v \in \child{u}}|\lambda_v|^2 < \infty;$ if
this is the case, then $\|\slam\|^2=\sup_{u\in V}
\|\slam e_u\|^2 = \sup_{u\in V} \sum_{v \in
\child{u}}|\lambda_v|^2.$
   \end{enumerate}
Moreover, if $\slam \in \B(\ell^2(V))$, then
   \begin{enumerate}
   \item[(iii)]
$\slam^* e_u = \bar{\lambda}_ue_{\parent{u}}$ if $u\in
V^{\circ}$ and $S^*_{\lambda}e_u=0$ otherwise,
   \item[(iv)] $\ker{\slam^*} =
   \begin{cases}
\langle e_{\rot} \rangle \oplus \bigoplus_{u \in
V^\prime} \big(\ell^2(\child{u}) \ominus \langle
\lambdab^u \rangle\big) & \text{if $\tcal$ is rooted,}
   \\[.5ex]
\bigoplus_{u \in V^\prime} \big(\ell^2(\child{u})
\ominus \langle \lambdab^u \rangle\big) &
\text{otherwise,}
   \end{cases}$
\\[1ex]
where $\lambdab^u
\in \ell^2(\child{u})$ is given by $\lambdab^u\colon
\child{u} \ni v \to \lambda_v \in \C$,
   \item[(v)]
$|\slam| e_u = \|\slam e_u\|e_u$ for all $u\in V,$
   \item[(vi)] $\triangle_{\slam} (e_u) = (\|\slam e_u\|^2 - 1)e_u$
for every $u\in V,$
   \item[(vii)]
   $
\triangle_{\slam^*} (e_u) =
   \begin{cases} \big(\sum_{v\in \child{\parent{u}}}
\lambda_v \bar{\lambda}_u e_v\big) - e_u & \text{if }
u \in V^{\circ},
   \\    - e_u & \text{if $\tcal$ is rooted
and } u=\omega,
   \end{cases}
   $
   \item[(viii)]
$\slam$ is analytic $($and thus completely
non-unitary$)$ if $\tcal$ is rooted.
   \end{enumerate}
   \end{lemma}
   \begin{proof}
The assertions (i)-(v) follow from \cite[Propositions~
3.1.3, 3.1.8, 3.4.1, 3.4.3 and 3.5.1]{JJS}. The
assertion (vi) can be deduced from (v), while the
assertion (vii) can be inferred from (i) and (iii). To
prove the assertion (viii), assume that $\tcal$ is
rooted. It follows from \cite[Corollary~ 2.1.5 and
Lemma 6.1.1]{JJS} that
   \begin{align*}
\slam^{n}(\ell^2(V)) \subseteq \chi_{\varOmega_n}
\cdot \ell^2(V), \quad n\in \zbb_+,
   \end{align*}
where $\varOmega_n=\bigsqcup_{j=n}^{\infty}
\childn{n}{\rot}$ for $n\in \zbb_+$. This implies that
$\bigcap_{n=0}^{\infty} \slam^{n}(\ell^2(V)) = \{0\}$,
which means that $\slam$ is analytic and so completely
non-unitary.
   \end{proof}
For a weighted shift $\slam\in \B(\ell^2(V))$ on
$\tcal$, we define $d_{\slam}\colon V \times \zbb_+
\to \rbb_+$ by
   \begin{align} \label{defdun}
d_{\slam}(u,n) = \|\slam^n e_u\|^2, \quad u \in V,\, n
\in \zbb_+.
   \end{align}
We show that $\slam^{*n}\slam^{n}$ is a diagonal
operator with respect to the orthonormal basis
$\{e_u\}_{u\in V}$ with diagonal elements
$\{d_{\slam}(u,n)\}_{u\in V}$,
   \begin{lemma} \label{powers}
Let $\slam \in \B(\ell^2(V))$ be a weighted shift on
$\tcal.$ Then
   \begin{align} \label{dun}
\slam^{*n} \slam^{n} e_u = d_{\slam}(u,n) e_u, \quad u
\in V, \, n\in \zbb_+.
   \end{align}
The function $d_{\slam}$ satisfies the following
recurrence relation{\em :}
   \begin{align}   \label{snssn1}
d_{\slam}(u,0) & =1, \quad u\in V,
   \\  \label{snssn2}
d_{\slam}(u,n+1) & = \sum_{v\in \child{u}}
|\lambda_v|^2 d_{\slam}(v,n), \quad u \in V, \, n \in
\zbb_+.
   \end{align}
   \end{lemma}
   \begin{proof}
We will use induction on $n$. The case of $n=0$ is
obvious. Assume that \eqref{dun} holds for a fixed
$n\in \zbb_+$. Then, by Lemma~ \ref{basicws}, we have
   \allowdisplaybreaks
   \begin{align}  \notag
\slam^{*(n+1)}\slam^{(n+1)}e_{u} & = \slam^{*} \sum_{v
\in \child{u}} \lambda_v \slam^{*n} \slam ^{n} e_{v}
   \\ \notag
&\hspace{-.3ex}\overset{(*)}= \sum_{v \in \child{u}}
\lambda_v d_{\slam}(v,n) \slam^{*} e_{v}
   \\ \notag
& = \sum_{v \in \child{u}} |\lambda_v|^2
d_{\slam}(v,n) e_{u}, \quad u \in V,
   \end{align}
where ($*$) is due to the induction hypothesis. This
completes the proof.
   \end{proof}
Given a weighted shift $\slam \in \B(\ell^2(V))$ on
$\tcal$, we~ set
   \begin{align*}
\{\lambdab\neq 0\}=\{v \in V^{\circ} \colon \lambda_v
\neq 0\} \quad \text{and} \quad V_{\lambdab}^+ =
\{u\in V\colon \|\slam e_u\| > 0\}.
   \end{align*}
It follows from Lemma~ \ref{basicws}(i) that
   \begin{align}  \label{parv}
V_{\lambdab}^+ = \big\{u\in V \colon \child{u} \cap
\{\lambdab\neq 0\} \neq \emptyset\big\} = \{u\in
V'\colon \lambdab^u \neq 0\} = \parent{\{\lambdab\neq
0\}}.
   \end{align}
Note that if $V_{\lambdab}^+=V$, then $\tcal$ is
leafless (but not conversely), and if $\tcal$ is
leafless and $\{\lambdab\neq 0\}=V^{\circ}$, then
$V_{\lambdab}^+=V$ (but not conversely).

Now we show that the operation of taking Cauchy dual
is an inner operation in the class of weighted shifts
on directed trees.
   \begin{lemma}\label{cisws}
Let $\slam \in \B(\ell^2(V))$ be a left-invertible
weighted shift on $\tcal$ with weights
$\lambdab=\{\lambda_v\}_{v\in V^{\circ}}$. Then
$V_{\lambdab}^+ = V$ and the Cauchy dual $\slam'$ of
$\slam$ is a weighted shift on $\tcal$ with weights
$\big\{\lambda_v\|\slam
e_{\parent{v}}\|^{-2}\big\}_{v\in V^{\circ}}$.
   \end{lemma}
   \begin{proof}
In view of \cite[Proposition~ 3.4.3(iv)]{JJS},
$(\slam^*\slam f)(u) = \|\slam e_u\|^2 f(u)$ for all
$u \in V$ and $f \in \ell^2(V)$. Since, by the
left-invertibility of $\slam$, $\slam^*\slam$ is
invertible in $\B(\ell^2(V))$, we deduce that
$V_{\lambdab}^+ = V$. Clearly, $((\slam^*\slam)^{-1}
f)(u) = \|\slam e_u\|^{-2} f(u)$ for all $u \in V$ and
$f \in \ell^2(V)$. This and the definition of $\slam$
complete the proof.
   \end{proof}
The question of when a weighted shift on a directed
tree satisfies the kernel condition has the following
explicit answer.
   \begin{lemma} \label{kcfws}
Let $\slam \in \B(\ell^2(V))$ be a weighted shift on
$\tcal$. Then the following conditions are
equivalent{\em :}
   \begin{enumerate}
   \item[(i)] $\slam^*\slam(\ker{\slam^*}) \subseteq
\ker{\slam^*}$,
   \item[(ii)] there exists a family $\{\alpha_v\}_{v\in V_{\lambdab}^+}
\subseteq \rbb_+$ such that
   \begin{align*}
\|\slam e_u\|=\alpha_{\parent{u}}, \quad u \in
\{\lambdab\neq 0\}.
   \end{align*}
   \end{enumerate}
Moreover, if $\tcal$ is leafless and $\slam$ has
nonzero weights, then {\em (i)} is equivalent to
   \begin{enumerate}
   \item[(iii)] there exists a family $\{\alpha_v\}_{v\in V}
\subseteq \rbb_+$ such that
   \begin{align}  \label{hypo+}
\|\slam e_u\|=\alpha_{\parent{u}}, \quad u \in
V^{\circ}.
   \end{align}
   \end{enumerate}
   \end{lemma}
   \begin{proof}
Given $v\in V$, we denote by $M_v$ the operator in
$\ell^2(\child{v})$ of multiplication by the function
$\child{v} \ni u\mapsto \|\slam {e_u}\|^2 \in \rbb_+.$
It follows from Lemma~ \ref{basicws}(ii) that $M_v \in
\B(\ell^2(\child{v}))$. Using \cite[Proposition~
2.1.2]{JJS} and Lemma~ \ref{basicws}(v) we get
   \begin{align*}
\slam^*\slam =
   \begin{cases}
\bigoplus_{v\in V'} M_v & \text{if $\tcal$ is
rootless,}
   \\[1ex]
\|\slam e_{\rot}\|^2 \cdot I_{\langle e_{\rot}
\rangle} \oplus \bigoplus_{v\in V'} M_v & \text{if
$\tcal$ is rooted.}
   \end{cases}
   \end{align*}
Hence, by Lemma~ \ref{basicws}(iv), the condition (i)
holds if and only if
   \begin{align} \label{num-r}
M_{v}\big(\ell^2(\child{v}) \ominus \langle
\lambdab^{v}\rangle\big) \subseteq \ell^2(\child{v})
\ominus \langle \lambdab^{v}\rangle, \quad v\in
V^{\prime}.
   \end{align}
Since $M_v=M_v^*$, \eqref{num-r} holds if and only if
$M_v(\langle \lambdab^{v}\rangle) \subseteq \langle
\lambdab^{v}\rangle$ for all $v\in V^{\prime}$, or
equivalently, if and only if $M_v(\langle
\lambdab^{v}\rangle) \subseteq \langle
\lambdab^{v}\rangle$ for all $v\in V_{\lambdab}^+$
(see \eqref{parv}), and the latter is equivalent to
(ii).

The ``moreover'' part is obvious due to \eqref{parv}
and the equivalence (i)$\Leftrightarrow$(ii).
   \end{proof}
   $2$-isometric weighted shifts on directed trees can
be characterized as follows.
   \begin{lemma} \label{2-is-sl}
A weighted shift $\slam \in \B(\ell^2(V))$ on $\tcal$
is a $2$-isometry if and only if either of the
following two equivalent conditions holds{\em :}
   \begin{gather} \label{char2iso}
1- 2 \|\slam e_u\|^2 + \sum_{v \in
\child{u}}|\lambda_v|^2 \|\slam e_v\|^2=0, \quad u \in
V,
   \\  \label{char2iso2}
\sum_{v\in \child{u}} |\lambda_v|^2 (2-\|\slam
e_v\|^2)=1, \quad u \in V.
   \end{gather}
If $\slam$ is a $2$-isometry, then $\|\slam e_u\| \Ge
1$ for all $u\in V,$ $V_{\lambdab}^+ = V$ and $\tcal$
is leafless.
   \end{lemma}
   \begin{proof}
Using \eqref{dun}, \eqref{snssn1} and \eqref{snssn2}
(see Lemma~ \ref{powers}), we deduce that $\slam$ is
$2$-isometric if and only if \eqref{char2iso} holds.
By Lemma~ \ref{basicws}(i), the conditions
\eqref{char2iso} and \eqref{char2iso2} are equivalent
(note that all series appearing in \eqref{char2iso}
and \eqref{char2iso2} are convergent).

The ``moreover part'' follows from \cite[Lemma 1]{R-0}
and Lemma~ \ref{basicws}(viii).
   \end{proof}
   \begin{remark} \label{codla2}
Let $\slam \in \B(\ell^2(V))$ be a $2$-isometric
weighted shift on a directed tree $\tcal$ with nonzero
weights $\lambdab=\{\lambda_v\}_{v \in V^{\circ}}$.
Since $\{\lambdab\neq 0\}=V^{\circ}$ and, by Lemma~
\ref{2-is-sl}, $\tcal$ is leafless, we infer from
Lemma~ \ref{kcfws} that $\slam$ satisfies the kernel
condition if and only if \eqref{hypo+} holds for some
$\{\alpha_v\}_{v\in V} \subseteq \rbb_+$.
   \hfill $\diamondsuit$
   \end{remark}
   Now we characterize $2$-isometric weighted shifts
on rooted directed trees which satisfy the condition
\eqref{hypo+}.
   \begin{lemma}\label{alpha2}
Suppose $\slam\in \B(\ell^2(V))$ is a weighted shift
on a rooted directed tree $\mathscr T$ which satisfies
the condition \eqref{hypo+} for some
$\{\alpha_v\}_{v\in V} \subseteq \rbb_+.$ Then the
following conditions are equivalent{\em :}
   \begin{enumerate}
   \item[(i)] $\slam$ is a $2$-isometry,
   \item[(ii)] $1 - 2 \|\slam e_{\rot}\|^2 +
\alpha_{\rot}^2 \|\slam e_{\rot}\|^2 =0$ and $1- 2
\alpha^2_{\parent{u}} + \alpha^2_u
\alpha^2_{\parent{u}} = 0$ for every $u \in
V^{\circ}.$
   \end{enumerate}
Moreover, if $\slam$ is a $2$-isometry, then $($see
\eqref{xin}$)$
   \begin{enumerate}
   \item[(iii)] $\|\slam e_{\rot}\| \Ge 1$ and\/\footnote{\;This
implies that $\alpha_u \in [1,\sqrt{2}\,)$ for all
$u\in V.$} $\alpha_u = \xi_{n+1}(\|\slam e_{\rot}\|)$
for all $u \in \childn{n}{\rot}$ and $n\in \zbb_+$,
   \item[(iv)] $\slam$ is an isometry if and only if
$\|\slam e_v\|=1$ for some $v\in V.$
   \end{enumerate}
   \end{lemma}
   \begin{proof}
The equivalence (i)$\Leftrightarrow$(ii) is a direct
consequence of \eqref{hypo+} and Lemmata~
\ref{basicws}(i) and \ref{2-is-sl}.

To prove the ``moreover'' part, assume that $\slam$ is
a $2$-isometry.

(iii) By \cite[Lemma 1]{R-0}, $\|\slam e_{\rot}\| \Ge
1$. We will use induction to prove that
   \begin{align} \label{indn0}
\alpha_u = \xi_{n+1}(\|\slam e_{\rot}\|), \quad u \in
\childn{n}{\rot},
   \end{align}
for every $n\in \zbb_+$. The case of $n=0$ follows
from the first equality in (ii). Assume that
\eqref{indn0} holds for a fixed $n\in \zbb_+$. Take $u
\in \childn{n+1}{\rot}$. Then, by \eqref{ind1},
$\parent{u} \in \childn{n}{\rot}$. It follows from the
induction hypothesis that
   \begin{align} \label{ind2}
\alpha_{\parent{u}} = \xi_{n+1}(\|\slam e_{\rot}\|)
\Ge 1.
   \end{align}
Using the second equation in (ii) and Lemma~
\ref{xin11}(ii), we get
   \begin{align*}
\alpha_u = \xi_1(\alpha_{\parent{u}})
\overset{\eqref{ind2}} = \xi_1(\xi_{n+1}(\|\slam
e_{\rot}\|)) = \xi_{n+2}(\|\slam e_{\rot}\|),
   \end{align*}
which completes the induction argument. Hence (iii)
holds.

(iv) Only the ``if'' part needs proof. Note that by
Lemma~ \ref{2-is-sl}, $\tcal$ is leafless. If $\|\slam
e_{\rot}\| = 1$, then by (iii) we have
   \begin{align*}
1=\xi_1(\|\slam e_{\rot}\|)=\alpha_{\rot}
\overset{\eqref{hypo+}} = \|\slam e_{\tilde v}\|,
\quad \tilde v \in \child{\rot}.
   \end{align*}
Hence, without loss of generality we can assume that
$\|\slam e_v\| = 1$ for some $v\in V^{\circ}$. Set
$u=\parent{v}$. By \eqref{ind1}, there exists $n\in
\zbb_+$ such that $v \in \childn{n+1}{\rot}$ and thus
$u \in \childn{n}{\rot}$. Then
   \begin{align*}
\xi_{n+1}(\|\slam e_{\rot}\|)
\overset{\mathrm{(iii)}}= \alpha_{u}
\overset{\eqref{hypo+}}= \|\slam e_v\| = 1,
   \end{align*}
which implies that $\|\slam e_{\rot}\|=1.$ By
\eqref{ind1} and (iii), $\alpha_w=1$ for all $w \in
V.$ Hence, in view of \eqref{hypo+}, $\|\slam e_w\|=1$
for all $w \in V.$ This combined with Lemma~
\ref{basicws}(vi) shows that $\slam$ is an isometry.
This completes the proof.
   \end{proof}
   \begin{proposition} \label{2iso-kc}
If $\slam\in \B(\ell^2(V))$ is a weighted shift on a
rooted directed tree $\tcal$, then the following
conditions are equivalent{\em :}
   \begin{enumerate}
   \item[(i)] $\slam$ is a $2$-isometry  satisfying
the condition \eqref{hypo+} for some
$\{\alpha_v\}_{v\in V} \subseteq \rbb_+,$
   \item[(ii)] $\|\slam e_{\rot}\| \Ge 1$ and
$\|\slam e_v\| = \xi_n(\|\slam e_{\rot}\|)$ for all $v
\in \childn{n}{\rot}$ and $n\in \zbb_+$.
   \end{enumerate}
   \end{proposition}
   \begin{proof}
The implication (i)$\Rightarrow$(ii) follows from
Lemma~ \ref{alpha2}(iii), \eqref{hypo+} and
\eqref{ind1}. To prove the reverse implication, define
$\{\alpha_v\}_{v\in V} \subseteq \rbb_+$ by $\alpha_u
= \xi_{n+1}(\|\slam e_{\rot}\|)$ for all $u \in
\childn{n}{\rot}$ and $n\in \zbb_+$, and verify, using
\eqref{ind1}, that the conditions \eqref{hypo+} and
(ii) of Lemma~ \ref{alpha2} are satisfied. Hence, by
this proposition, (i) holds.
   \end{proof}
Below, we will show that $2$-isometric weighted shifts
on rootless directed trees satisfying \eqref{hypo+}
must be isometric (clearly, each isometric weighted
shift on a directed tree satisfies \eqref{hypo+}).
This is somehow related to \cite[Theorem~
7.2.1(iii)]{JJS}.
   \begin{proposition} \label{bilat}
Let $\slam\in \B(\ell^2(V))$ be a $2$-isometric
weighted shift on a rootless directed tree $\tcal$,
which satisfies the condition \eqref{hypo+} for some
$\{\alpha_v\}_{v\in V} \subseteq \rbb_+.$ Then $\slam$
is an isometry.
   \end{proposition}
   \begin{proof}
In view of Lemma~ \ref{basicws}(vi), it suffices to
show that $\|\slam e_v\|=1$ for every $v\in V$. Fix
$v\in V$. Since $\tcal$ is rootless and leafless (see
Lemma~ \ref{2-is-sl}), an induction argument shows
that there exists a (necessarily injective) sequence
$\{v_n\}_{n = -\infty}^{\infty}\subseteq V$ such that
$v_0=v$ and $v_{n}=\parent{v_{n+1}}$ for all $n\in
\zbb$. Set $\beta_n=\|\slam e_{v_n}\|$ for $n\in
\zbb$. Clearly, $\{\beta_n\}_{n\in\zbb} \subseteq
[0,\|\slam\|]$. According to \eqref{hypo+},
\eqref{char2iso} and Lemma~ \ref{basicws}(i), we have
   \begin{align*}
1- 2 \beta_n^2 + \beta_n^2 \beta_{n+1}^2 = 0, \quad
n\in \zbb.
   \end{align*}
Hence, by Lemma~ \ref{2-is-sl}, the bilateral weighted
shift $W$ in $\ell^2(\zbb)$ with weights
$\{\beta_n\}_{n\in\zbb}$ is a $2$-isometry with dense
range. Since $W$ is left-invertible, we deduce that
$W$ is invertible in $\B(\ell^2(\zbb))$. Therefore, by
\cite[Proposition~ 1.23]{Ag-St} (see also
\cite[Remark~ 3.4]{Sho-Ath}), $W$ is unitary. This
implies that $\beta_n=1$ for all $n\in \zbb$. In
particular, $\|\slam e_v\|=\|\slam
e_{v_0}\|=\beta_0=1$, which completes the proof.
   \end{proof}
   We conclude this section by showing that Brownian
isometric weighted shifts on rooted directed trees are
isometric (cf.\ Proposition~ \ref{stwfa}).
   \begin{proposition}\label{briai}
Let $\slam\in \B(\ell^2(V))$ be a Brownian isometric
weighted shift on a rooted directed tree $\mathscr T$.
Then $\slam$ is an isometry.
   \end{proposition}
   \begin{proof}
We split the proof into a few steps.

{\sc Step} 1. If $u\in V$ is such that either
$u=\omega$ or $u\in V^{\circ}$ and $\lambda_u=0,$ then
$\|\slam e_u\|=1.$

Indeed, it follows from \eqref{defbi} and the
assertions (vi) and (vii) of Lemma~ \ref{basicws} that
   \begin{align*}
0=\triangle_{\slam} \triangle_{\slam^*}
\triangle_{\slam} (e_u) & = (\|\slam
e_u\|^2-1)\triangle_{\slam} \triangle_{\slam^*} (e_u)
   \\
& = -(\|\slam e_u\|^2-1)\triangle_{\slam} (e_u)
   \\
& = -(\|\slam e_u\|^2-1)^2 e_u,
   \end{align*}
which means that $\|\slam e_u\|=1.$

{\sc Step} 2. If $u \in V$ is such that $\|\slam
e_u\|=1,$ then $\|\slam e_v\|=1$ for all $v\in
\child{u}.$

Indeed, by \eqref{char2iso} and Lemma~
\ref{basicws}(i), we see that
   \begin{align*}
\sum_{v \in \child{u}}|\lambda_v|^2 \|\slam e_v\|^2 =
1 = \|\slam e_u\|^2 = \sum_{v \in \child{u}}
|\lambda_v|^2.
   \end{align*}
Hence, we have
   \begin{align} \label{sumeqo}
\sum_{v \in \child{u}}|\lambda_v|^2 (\|\slam e_v\|^2
-1) =0.
   \end{align}
Since by Lemma~ \ref{2-is-sl}, $\|\slam e_v\|^2 - 1
\Ge 0$ for all $v\in V,$ we infer from \eqref{sumeqo}
that $\|\slam e_v\| = 1$ for all $v\in \child{u}$ such
that $\lambda_v\neq 0.$ On the other hand, if
$\lambda_v=0$ for some $v\in \child{u},$ then by Step
1, $\|\slam e_v\|=1,$ which completes the proof of
Step 2.

Finally, induction together with \eqref{ind1} and
Steps 1 and 2 shows that $\|\slam e_u\|=1$ for all
$u\in V,$ which implies that $\slam$ is an isometry
(see Lemma~ \ref{basicws}(vi)).
    \end{proof}
   \section{\label{Sec10}The Cauchy dual subnormality
problem via perturbed kernel condition}
   Remark~ \ref{codla2} suggests considering a wider
class of $2$-isometric weighted shifts on directed
trees which satisfy a less restrictive condition than
\eqref{hypo+}. In this section, we will discuss the
question of subnormality of the Cauchy dual of a
$2$-isometric weighted shift $\slam\in \B(\ell^2(V))$
on a rooted directed tree $\mathscr T$ for which there
exist $k\in \nbb$ and a family $\{\alpha_v\}_{v\in
\des{\childn{k}{\rot}}} \subseteq \rbb_+$ such that
   \begin{align} \label{hypok}
\|\slam e_u\|=\alpha_{\parent{u}}, \quad u \in
\des{\childn{k+1}{\rot}}.
   \end{align}
A complete answer to this question is given in
Theorem~ \ref{main2}. This enables us to solve the
Cauchy dual subnormality problem in the negative (see
Example~ \ref{glowny}; see also Example~ \ref{przadj}
for the case of adjacency operators). For a pictorial
comparison of the conditions \eqref{hypo+} and
\eqref{hypok} in the case of $k=1$, we refer the
reader to Figure~ \ref{Fig1a}. The quantities $\|\slam
e_v\|,$ $v\in \childn{2}{\omega}$, appearing therein
can be calculated by using \eqref{hypo} and
\eqref{2-iso-0}.

We begin by establishing an explicit formula for
$d_{\slam^{\prime}}$ (see Lemma~ \ref{cisws} and
\eqref{defdun}), where $\slam$ is a $2$-isometric
weighted shift on a rooted directed tree which
satisfies the condition \eqref{hypok} for $k=1$.
   \begin{figure}[t]
   \begin{center}
\subfigure {
\includegraphics[scale=0.39]{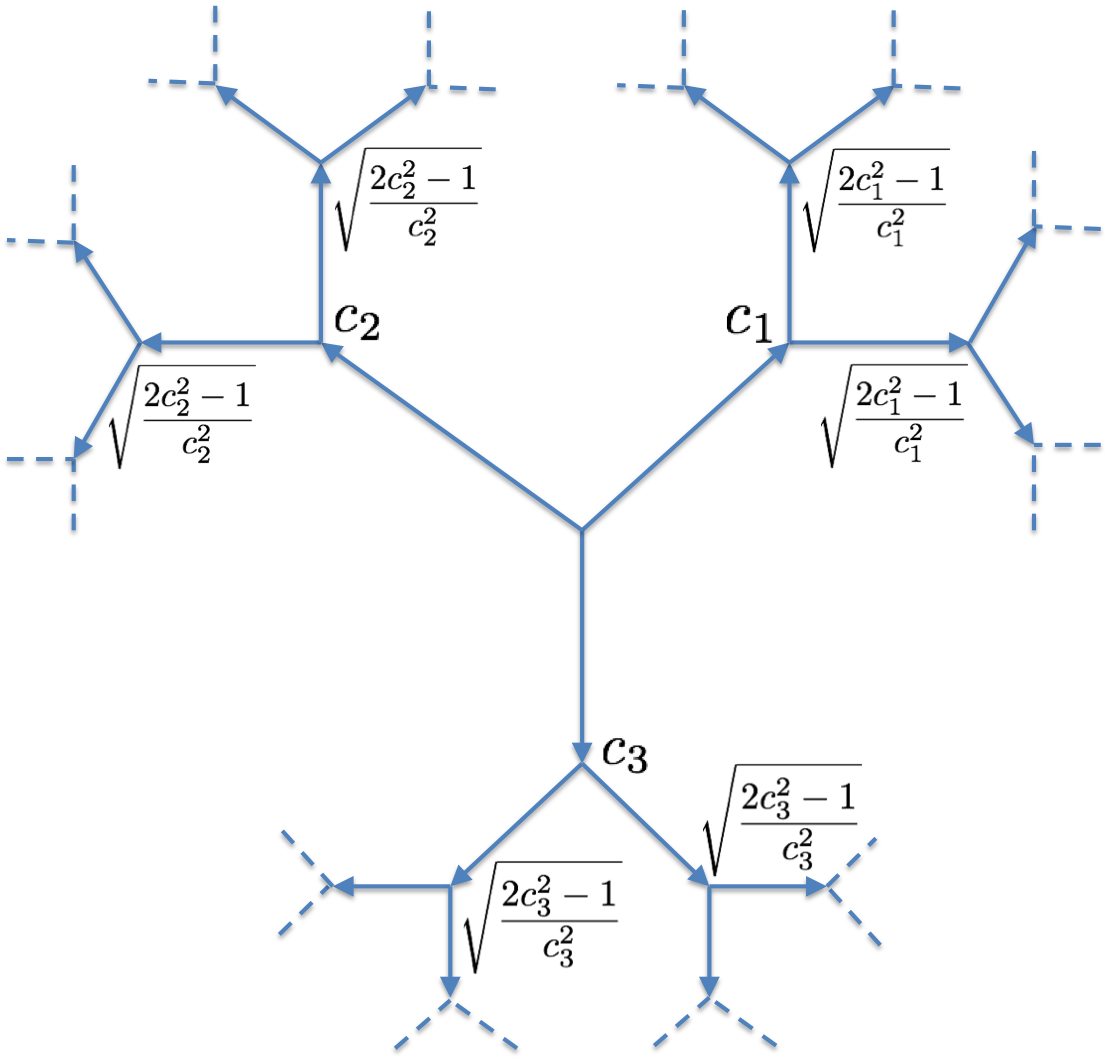}
} \caption{A weighted shift $\slam$ on a rooted
directed tree $\tcal$ which satisfies \eqref{hypok}
for $k=1$; it satisfies \eqref{hypo+} exclusively when
$c_1=c_2=c_3$ ($\Vert \slam e_v \Vert$ is the label of
a vertex $v$).} \label{Fig1a}
   \end{center}
   \end{figure}
   \begin{lemma} \label{d-graph}
Let $\slam\in \B(\ell^2(V))$ be a $2$-isometric
weighted shift on a rooted directed tree $\mathscr T$
such that
   \begin{align}  \label{hypo}
\|\slam e_u\|=\alpha_{\parent{u}}, \quad u \in
\des{\childn{2}{\rot}},
   \end{align}
for some $\{\alpha_v\}_{v\in V^{\circ}} \subseteq
\rbb_+$. Then $\tcal$ is leafless, $\|\slam e_u\| \Ge
1$ for every $u\in V$ and
   \begin{gather} \label{moment-a}
d_{\slam^{\prime}}(\rot,n) = \frac{1}{\|\slam
e_{\rot}\|^4} \sum_{v \in \child{\rot}}
\frac{|\lambda_v|^2}{(n-1)\| \slam e_v\|^2 - (n-2)},
\quad n\in \nbb,
   \\ \label{moment-b}
d_{\slam^{\prime}}(v,n) = \frac{\|\slam
e_v\|^{-2}}{(n-1) \alpha^2_v - (n-2)}, \quad v \in
V^{\circ}, \, n\in \nbb.
   \end{gather}
   \end{lemma}
   \begin{proof}
By Lemma~ \ref{2-is-sl}, $\tcal$ is leafless and
$\|\slam e_u\| \Ge 1$ for all $u\in V$. Hence, by
\eqref{hypo}, the expressions appearing in
\eqref{moment-a} and \eqref{moment-b} make sense. It
follows from Lemma~ \ref{2-is-sl} and \eqref{hypo}
that
   \begin{gather} \label{2-iso-0}
   1- 2 \|\slam e_u\|^2 + \alpha^2_u
\|\slam e_u\|^2 = 0, \quad u \in \child{\rot},
   \\ \label{2-iso}
1- 2 \alpha^2_{\parent{u}} + \alpha^2_u
\alpha^2_{\parent{u}} = 0, \quad u \in V^{\circ}
\setminus
\child{\rot}\overset{\eqref{ind1}}=\des{\childn{2}{\rot}}.
   \end{gather}

Below we shall use Lemmata~ \ref{powers} and
\ref{cisws} without explicitly mentioning them. We
prove the equality \eqref{moment-b} by induction on
$n$. That it holds for $n=1$ follows from Lemma~
\ref{basicws}(i). Assume it holds for a fixed $n\in
\nbb$. Then we have
   \allowdisplaybreaks
   \begin{align*}  d_{\slam^{\prime}}(u,n+1) & = \sum_{v \in \child{u}}
\frac{|\lambda_v|^2} {\|\slam e_{\parent{v}}\|^4} \,
d_{\slam^{\prime}}(v,n)
   \\ &\hspace{-.3ex}\overset{(*)}= \frac{1}{\|\slam
e_{u}\|^4} \sum_{v \in \child{u}}
\frac{|\lambda_v|^2}{\|\slam e_{v}\|^2((n-1)
\alpha^2_v -(n-2))}
   \\ &\hspace{-.7ex}\overset{\eqref{hypo}}=
\frac{1}{\|\slam e_{u}\|^4} \sum_{v \in \child{u}}
\frac{|\lambda_v|^2}{\alpha_u^2((n-1) \alpha^2_v
-(n-2))}
   \\  & \hspace{-.7ex}\overset{\eqref{2-iso}}= \frac{\|\slam
e_u\|^{-2}}{n \alpha^2_u - (n-1)}, \quad u\in
V^{\circ},
   \end{align*}
where ($*$) is due to the induction hypothesis. Hence,
\eqref{moment-b} holds.

The equality \eqref{moment-a} will be deduced from
\eqref{moment-b}. The case of $n=1$ follows from
Lemma~ \ref{basicws}(i). Let us fix an integer $n\Ge
2$. Then we have \allowdisplaybreaks
   \begin{align}  \notag
d_{\slam^{\prime}}(\rot, n) & = \frac{1}{\|\slam e_{\rot}\|^4} \sum_{v \in \child{\rot}} |\lambda_v|^2 d_{\slam^{\prime}}(v, n-1)
   \\ \notag
&\hspace{-.7ex} \overset{\eqref{moment-b}}= \frac{1}{\|\slam e_{\rot}\|^4}\sum_{v
\in \child{\rot}} \frac{\|\slam
e_v\|^{-2} |\lambda_v|^2}{(n-2) \alpha^2_v -(n-3)}
   \\ \notag
&\hspace{-.7ex} \overset{\eqref{2-iso-0}}= \frac{1}{\|\slam e_{\rot}\|^4} \sum_{v \in
\child{\rot}} \frac{\|\slam
e_v\|^{-2} |\lambda_v|^2}{(n-2) (2- \|\slam
e_v\|^{-2}) -(n-3)},
   \\ \notag
&= \frac{1}{\|\slam e_{\rot}\|^4}\sum_{v \in \child{\rot}}
\frac{|\lambda_v|^2}{(n-1)\| \slam e_v\|^2 - (n-2)},
   \end{align}
which completes the proof.
   \end{proof}
   \begin{remark}
Note that the formula \eqref{moment-b} can be derived
from \eqref{cdformula} by using the fact that the
operator $T_{\lambdab}:=\slam |_{\mathcal M^{\perp}}$,
where $\mathcal M = \langle e_{\rot} \rangle,$ is a
$2$-isometry which satisfies the kernel condition and
the assumptions of Proposition \ref{restrcd} with
$\mathcal L=\mathcal M^{\perp}$. One may refer to
$\slam $ as a {\it rank one $2$-isometric extension}
of $T_{\lambdab}.$ We will show in Example
\ref{glowny} that the Cauchy dual subnormality problem
has a negative solution even for rank one
$2$-isometric extensions of $2$-isometries which
satisfy the kernel condition. \hfill $\diamondsuit$
   \end{remark}
Recall a criterion for a Stieltjes moment sequence to
have a backward extension.
   \begin{lemma}[\mbox{\cite[Lemma
6.1.2]{JJS}}] \label{bext} Let
$\{\gamma_n\}_{n=1}^\infty \subseteq \rbb_+$ be a
sequence such that $\{\gamma_{n+1}\}_{n=0}^\infty$ is
a Stieltjes moment sequence. Set $\gamma_0=1$. Then
the following conditions are equivalent{\em :}
   \begin{enumerate}
   \item[(i)] $\{\gamma_n\}_{n=0}^\infty$ is a
Stieltjes moment sequence,
   \item[(ii)] there exists a representing measure $\mu$
of $\{\gamma_{n+1}\}_{n=0}^\infty$ concentrated on
$\rbb_+$ such that\/\footnote{\;We adhere to the
convention that $\frac 1 0 := \infty$. Hence,
$\int_{\rbb_+} \frac 1 t d \mu(t) < \infty$ implies
$\mu(\{0\})=0$.} $\int_{\rbb_+} \frac 1 t d \mu(t) \Le
1$.
   \end{enumerate}
If $\mu$ is as in {\em (ii)}, then the positive Borel
measure $\nu$ on $\rbb$ defined by
   \begin{align*}
\nu(\varDelta) = \int_{\varDelta} \frac 1 t d \mu(t) +
\bigg(1-\int_{\rbb_+} \frac 1 t d \mu(t)\bigg)
\delta_0(\varDelta), \quad \varDelta \in \borel{\rbb},
   \end{align*}
is a representing measure of
$\{\gamma_n\}_{n=0}^\infty$ concentrated on $\rbb_+;$
moreover, $\nu(\{0\})=0$ if and only if $\int_{\rbb_+}
\frac 1 t d \mu(t)=1$.
   \end{lemma}
The next lemma is an essential ingredient of the proof
of the implication (i)$\Rightarrow$(ii) of Theorem~
\ref{main2}. In fact, it covers the case of $k=1$ of
this implication.
   \begin{lemma} \label{d-graph2}
Let $\mathscr T$ and $\slam$ be as in Lemma~ {\em
\ref{d-graph}}. Assume that $\lambda_v \neq 0$ for
every $v\in \child{\rot}$. If the Cauchy dual
$\slam^\prime$ of $\slam$ is subnormal, then there
exists $\alpha \in \rbb_+$ such that $\|\slam e_v\| =
\alpha$ for all $v \in \child{\rot}.$
   \end{lemma}
   \begin{proof}
Assume that $\slam^{\prime}$ is subnormal. It follows
from \eqref{defdun} applied to $\slam^{\prime}$ and
\eqref{moment-a} that
   \begin{align*}
\|\slam^{\prime (n+1)} e_{\rot}\|^2 = \frac{1}{\|\slam
e_{\rot}\|^4} \sum_{v \in \child{\rot}}
\frac{|\lambda_v|^2}{1+n(\|\slam e_v\|^2 - 1)}, \quad
n\in \zbb_+.
   \end{align*}
This combined with Lemmata~ \ref{over} and
\ref{d-graph} (as well as with the Lebesgue monotone
convergence theorem) implies that $\|\slam e_v\| \Ge
1$ for all $v\in V$ and $\{\|\slam^{\prime (n+1)}
e_{\rot}\|^2\}_{n=0}^{\infty}$ is a Hausdorff moment
sequence with a (unique) representing measure $\rho$
given by
   \begin{align}  \label{rmiar}
\rho = \frac{1}{\|\slam e_{\rot}\|^4} \sum_{v \in
\child{\rot}} |\lambda_v|^2 \mu_{1,\|\slam e_v\|^2 -
1}.
   \end{align}
Since, by Lambert's theorem (see Theorem~ \ref{Lam}),
$\{\|\slam^{\prime n} e_{\rot}\|^2\}_{n=0}^{\infty}$
is a Stieltjes moment sequence, we infer from Lemma~
\ref{bext} that
   \begin{align} \label{nier1}
\int_{[0,1]} \frac{1}{t} d \rho(t) \Le 1.
   \end{align}
Set $\varSigma_1 = \{v \in \child{\rot}\colon \|\slam
e_v\|=1\}$ and $\varSigma_2 = \{v \in
\child{\rot}\colon \|\slam e_v\|>1\}$. Note that
$\child{\rot} = \varSigma_1 \sqcup \varSigma_2$ (the
disjoint sum). It follows from Lemma~ \ref{over} that
      \begin{align} \notag
\|\slam e_{\rot}\|^4 \int_{[0,1]} \frac{1}{t} & d
\rho(t) \overset{\eqref{rmiar}}= \sum_{v \in
\child{\rot}} |\lambda_v|^2 \int_{[0,1]} \frac{1}{t} d
\mu_{1,\|\slam e_v\|^2 - 1}(t)
   \\ \label{calka1}
&= \sum_{v \in \varSigma_1} |\lambda_v|^2 + \sum_{v
\in \varSigma_2} \frac{|\lambda_v|^2}{\|\slam e_v\|^2
- 1} \int_{[0,1]} t^{\frac{1}{\|\slam e_v\|^2 - 1} -
2} d t.
   \end{align}
Since, by assumption, $\lambda_v \neq 0$ for all $v\in
\varSigma_2$, the conditions \eqref{log},
\eqref{nier1} and \eqref{calka1} imply that $\|\slam
e_v\|^2 < 2$ for all $v\in \varSigma_2$, and thus for
all $v\in \child{\rot}$. Therefore, by \eqref{log} and
\eqref{calka1}, we have
   \begin{align*}
\|\slam e_{\rot}\|^4 \int_{[0,1]} \frac{1}{t} d
\rho(t) = \sum_{v \in \varSigma_1} |\lambda_v|^2 +
\sum_{v \in \varSigma_2}
\frac{|\lambda_v|^2}{2-\|\slam e_v\|^2} = \sum_{v \in
\child{\rot}} \frac{|\lambda_v|^2}{2-\|\slam e_v\|^2}.
   \end{align*}
This together with \eqref{nier1} yields
   \begin{align} \label{sum2-4}
\sum_{v \in \child{\rot}}
\frac{|\lambda_v|^2}{2-\|\slam e_v\|^2} \Le \|\slam
e_{\rot}\|^4.
   \end{align}
Now observe that
   \allowdisplaybreaks
   \begin{align*}
\sum_{v \in \child{\rot}}
\frac{|\lambda_v|^2}{2-\|\slam e_v\|^2} & \Le
\bigg(\sum_{v\in\child{\rot}} |\lambda_v|^2\bigg)^2
\quad \quad \quad \Big(\text{by \eqref{sum2-4} and
Lemma~ \ref{basicws}(i)}\Big)
   \\
& = \bigg(\sum_{v\in\child{\rot}}
\frac{|\lambda_v|^2(2-\|\slam e_v\|^2)}{2-\|\slam
e_v\|^2}\bigg)^2
   \\
& \hspace{-.3ex}\overset{(*)}\Le
\sum_{v\in\child{\rot}} \frac{|\lambda_v|^2(2-\|\slam
e_v\|^2)}{(2-\|\slam e_v\|^2)^2} \cdot
\sum_{v\in\child{\rot}} |\lambda_v|^2 (2-\|\slam
e_v\|^2)
   \\
& = \sum_{v\in\child{\rot}}
\frac{|\lambda_v|^2}{2-\|\slam e_v\|^2} \quad \quad
\quad (\text{apply \eqref{char2iso2} to $u=\rot$}),
   \end{align*}
where ($*$) follows from the Cauchy-Schwarz
inequality. Hence, equality holds in the
Cauchy-Schwarz inequality ($*$). This means that
$\Big\{\frac{\lambda_v\sqrt{2-\|\slam
e_v\|^2}}{2-\|\slam e_v\|^2}\Big\}_{v\in\child{\rot}}$
and $\Big\{\lambda_v\sqrt{2-\|\slam
e_v\|^2}\Big\}_{v\in\child{\rot}}$ are linearly
dependent vectors in $\ell^2(\child{\rot})$. Since the
weights $\{\lambda_v\}_{v\in \child{\rot}}$ are
nonzero, we deduce that there exists $\alpha \in
\rbb_+$ such that $\|\slam e_v\| = \alpha$ for every
$v \in \child{\rot},$ which completes the proof.
   \end{proof}
We are in a position to prove the main result of this
section, which can be thought of as a reconstruction
theorem.
   \begin{theorem} \label{main2}
Let $\slam\in \B(\ell^2(V))$ be a $2$-isometric
weighted shift on a rooted directed tree $\mathscr T,$
which satisfies \eqref{hypok} for some $k\in \nbb$ and
$\{\alpha_v\}_{v\in \des{\childn{k}{\rot}}} \subseteq
\rbb_+$. Let $\lambda_v \neq 0$ for all $v \in
\bigsqcup_{i=1}^k \childn{i} {\rot}$. Then the
following conditions are equivalent{\em :}
   \begin{enumerate}
   \item[(i)] the Cauchy dual $\slam^\prime$ of
$\slam$ is subnormal,
   \item[(ii)] there exists a family $\{\alpha_v\}_{v\in
\bigsqcup_{i=0}^{k-1} \childn{i} {\rot}} \subseteq
\rbb_+$ such that
   \begin{align*}
\|\slam e_u\|=\alpha_{\parent{u}}, \quad u \in
\bigsqcup_{i=1}^k \childn{i} {\rot},
   \end{align*}
   \item[(iii)] $\slam$ satisfies the condition {\em
(iii)} of Lemma~ {\em \ref{kcfws}},
   \item[(iv)] $\slam^*\slam(\ker{\slam^*}) \subseteq
\ker{\slam^*}$.
   \end{enumerate}
   \end{theorem}
   \begin{proof}
(i)$\Rightarrow$(ii) Fix $v\in \childn{k-1}{\rot}$.
Note that the space $\ell^2(\des{v})$ (which is
identified with the closed vector subspace
$\chi_{\des{v}}\cdot \ell^2(V)$ of $\ell^2(V)$) is
invariant for $\slam$ and $\slam|_{\ell^2(\des{v})}$
coincides with the weighted shift $S_{\lambdab^{|v
\rangle}}$ on the directed tree
$\tcal_{\des{v}}:=\big(\des{v}, (\des{v} \times
\des{v}) \cap E\big)$ with weights $\lambdab^{|v
\rangle}:=\{\lambda_u\}_{u\in \des{v}\setminus \{v\}}$
(see \cite[Proposition~ 2.1.8]{JJS} for more details).
It follows from \cite[Proposition~ 2.1.10]{JJS} and
the fact that $v$ is a root of $\tcal_{\des{v}}$ that
   \begin{gather}  \notag
\childs{\tcal_{\des{v}}}{v} = \child{v} \subseteq
\child{\childn{k-1}{\rot}} = \childn{k}{\rot},
   \\ \notag
\parents{\tcal_{\des{v}}}{u} = \parent{u} \text{ for all }
u \in \des{v}\setminus \{v\},
   \\ \label{relat1}
\dess{\tcal_{\des{v}}}{\childns{\tcal_{\des{v}}}{2}{v}}
= \des{\childn{2}{v}} \subseteq
\des{\childn{k+1}{\rot}}.
   \end{gather}
(The expressions $\parents{\tcal_{\des{v}}}{\cdot}$,
$\childs{\tcal_{\des{v}}}{\cdot}$,
$\childns{\tcal_{\des{v}}}{2}{\cdot}$ and
$\dess{\tcal_{\des{v}}}{\cdot}$ are understood with
respect to the directed subtree $\tcal_{\des{v}}$.)
This and \eqref{hypok} imply that $S_{\lambdab^{|v
\rangle}}$ satisfies the assumptions of Lemma~
\ref{d-graph2}. Applying Lemma~ \ref{basicws}(v) and
Proposition~ \ref{restrcd} to $T=\slam$ and $\mathcal
L= \ell^2(\des{v})$, we deduce that the Cauchy dual
$S_{\lambdab^{|v \rangle}}^\prime$ of $S_{\lambdab^{|v
\rangle}}$ is subnormal. Hence, by Lemma~
\ref{d-graph2}, there exists $\alpha_{v} \in \rbb_+$
such that (cf.\ \eqref{relat1})
   \begin{align*}
\|\slam e_u\|=\|S_{\lambdab^{|v \rangle}}
(e_u|_{\des{v}})\| = \alpha_{v} =
\alpha_{\parent{u}}\text{ for all } u\in
\childs{\tcal_{\des{v}}}{v} = \child{v}.
   \end{align*}
Summarizing, we have proved that there exists a family
$\{\alpha_v\}_{v\in \childn{k-1} {\rot}} \subseteq
\rbb_+$ such that $\|\slam e_u\|=\alpha_{\parent{u}}$
for all $u \in \child{v}$ and $v \in \childn{k-1}
{\rot}$. Since, by \cite[(2.2.6)]{BJJS},
$\childn{k}{\rot} = \bigsqcup_{v\in
\childn{k-1}{\rot}} \child{v}$, we see that $\|\slam
e_u\|=\alpha_{\parent{u}}$ for all $u\in
\childn{k}{\rot}$. Now, using reverse induction on
$k$, we conclude that (ii) holds.

Since, in view of \eqref{ind1}, the implication
(ii)$\Rightarrow$(iii) is obvious and the implications
(iii)$\Rightarrow$(iv) and (iv)$\Rightarrow$(i) are
direct consequences of Lemma~ \ref{kcfws} and Theorem~
\ref{cdsubn} respectively, the proof is complete.
   \end{proof}
We conclude this section by answering the Cauchy dual
subnormality problem in the negative. The
counterexample presented below is built over the
directed tree $\tcal_{2,0}$ (see Example~
\ref{trzykrow}(b)). It is easily seen that similar
counterexamples can be built over any directed tree of
the form $\tcal_{\eta,0}$, where $\eta \in
\{2,3,4,\ldots\} \cup \{\infty\}$.
   \begin{example} \label{glowny}
Let $y_1,y_2\in \rbb$ be such that $1 < y_1, y_2 <
\sqrt{2}$ and $y_1\neq y_2$. Then there exist positive
real numbers $x_1$ and $x_2$ such that
   \begin{align*}
\sum_{i=1}^2 x_i^2(2-y_i^2)=1 \qquad \text{(e.g.,
$x_i=\frac{1}{\sqrt{2(2-y_i^2)}}$ for $i=1,2$).}
   \end{align*}
Let $\slam$ be the weighted shift on $\tcal_{2,0}$
with weights $\lambdab=\{\lambda_v\}_{v\in
V_{2,0}^{\circ}}$ defined by
   \begin{align*}
\lambda_{i,j}=
   \begin{cases}
x_i & \text{if } j = 1,
   \\
\xi_{j-2}(y_i) & \text{if } j\Ge 2,
   \end{cases}
\qquad i = 1,2.
   \end{align*}
By Lemma~ \ref{basicws}(ii), $\slam \in
\B(\ell^2(V_{2,0}))$ and $\|\slam\|=\max\Big\{y_1,y_2,
\sqrt{\sum_{i=1}^2 x_i^2}\,\Big\}$. It is a matter of
routine to show that $\slam$ satisfies the condition
\eqref{char2iso2} and thus, by Lemma~ \ref{2-is-sl},
$\slam$ is a $2$-isometry. Moreover, by Lemma~
\ref{basicws}(viii), $\slam$ is completely
non-unitary. Note that $\slam$ satisfies the condition
\eqref{hypo} for $\{\alpha_v\}_{v\in V_{2,0}^{\circ}}
\subseteq \rbb_+$ given by $\alpha_{i,j} = \xi_j(y_i)$
for $i\in \{1,2\}$ and $j\in \{1,2,\dots\}$. Since the
weights of $\slam$ are nonzero and $\|\slam e_{1,1}\|
= y_1 \neq y_2 = \|\slam e_{2,1}\|$, we infer from
Theorem~ \ref{main2} with $k=1$ (see also Lemma~
\ref{d-graph2}) that the Cauchy dual $\slam^{\prime}$
of $\slam$ is not subnormal.
   \hfill $\diamondsuit$
   \end{example}
   \section{\label{Sec11}The Cauchy dual subnormality problem for
adjacency operators} In this section, we turn our
attention to weighted shifts on directed trees with
weights whose moduli are constant on $\child{u}$ for
every vertex $u$. This class of operators contains the
class of adjacency operators of directed trees; the
latter class plays an important role in graph theory
(see \cite{JJS} for more details). We prove that the
Cauchy dual of a $2$-isometric adjacency operator of a
directed tree is subnormal if the directed tree
satisfies certain degree constraints (see Theorem~
\ref{constant-t}). However, as shown in Example~
\ref{przadj} below, the Cauchy dual subnormality
problem has a negative solution even in the class of
adjacency operators of directed trees.

We begin by proving a preparatory lemma.
   \begin{lemma} \label{adja1}
Let $\slam \in \B(\ell^2(V))$ be a weighted shift on a
leafless and locally finite directed tree $\tcal$ such
that
   \begin{align} \label{beta}
|\lambda_{v}| = \beta_{\parent{v}}, \quad v\in
V^{\circ},
   \end{align}
for some $\{\beta_u\}_{v\in V} \subseteq (0,\infty)$.
Then the following statements hold{\em :}
   \begin{enumerate}
   \item[(i)] $\slam $ is a $2$-isometry if and only if
   \begin{align*}
\sum_{v \in \child{u}} \beta_{v}^2 \deg{v} =
\frac{2\beta_u^2 \deg{u} - 1}{\beta_u^2}, \quad u \in
V,
   \end{align*}
   \item[(ii)]
if $\slam$ is left-invertible, then
$d_{\slam^{\prime}}(u, 0)=1$ for all $u \in V$ and
   \begin{align*}
d_{\slam^{\prime}}(u, n+1) =
\frac{1}{{\beta^{2}_u}{(\deg{u})^{2}}} \sum_{v \in
\child{u}} d_{\slam^{\prime}}(v, n), \quad u \in V, \,
n \in \zbb_+,
   \end{align*}
where $d_{\slam^{\prime}}$ is given by \eqref{defdun}
with $\slam^{\prime}$ in place of $\slam$ $($see
Lemma~ {\em \ref{cisws}}$)$.
   \end{enumerate}
   \end{lemma}
   \begin{proof}
It follows from Lemma~ \ref{basicws}(i) and
\eqref{beta} that
   \begin{align} \label{normbeta}
\|\slam e_u\|^2 = \beta_u^2 \deg{u}, \quad u\in V.
   \end{align}
The statement (i) can be straightforwardly deduced
from \eqref{char2iso2}, \eqref{beta} and
\eqref{normbeta}. In turn, the statement (ii) can be
easily inferred from \eqref{snssn1} and \eqref{snssn2}
applied to $\slam^{\prime}$ by using \eqref{beta} and
\eqref{normbeta}.
   \end{proof}
By the {\em adjacency operator} of a directed tree
$\tcal,$ we understand the weighted shift $\slamj$ on
$\tcal$ all of whose weights are equal to $1$ (see
\cite[p.\ 1]{JJS} for more information). Note that in
general, adjacency operators may not even be densely
defined (see \cite[Proposition~ 3.1.3]{JJS}).

Below we describe some classes of $2$-isometric
adjacency operators including those satisfying the
kernel condition, Brownian isometries and
quasi-Brownian isometries (see Proposition~
\ref{stwfa}). The following preliminary result
characterizes isometric adjacency operators.
   \begin{lemma} \label{Hay2}
If $\slamj$ is the adjacency operator of a directed
tree $\tcal,$ then the following conditions are
equivalent{\em :}
   \begin{enumerate}
   \item[(i)] $\slamj$ is an isometry on $\ell^2(V)$,
   \item[(ii)] $\deg{u}=1$ for all $u\in V,$
   \item[(iii)] $\tcal$ is graph
isomorphic either to $\zbb_+$ or to $\zbb.$
   \end{enumerate}
   \end{lemma}
   \begin{proof}
That (i) and (ii) are equivalent follows from
\cite[Corollary~ 3.4.4]{JJS} and \eqref{normbeta}. In
turn, the equivalence (ii)$\Leftrightarrow$(iii) can
be deduced from \cite[Corollary~ 2.1.5 and
Proposition~ 2.1.6]{JJS}).
   \end{proof}
   Now we are in a position to describe the aforesaid
classes of $2$-isometric adjacency operators (see
Example~ \ref{trzykrow} for necessary definitions).
   \begin{proposition} \label{stwfa}
If $\slamj \in \B(\ell^2(V))$ is the adjacency
operator of a directed tree $\tcal,$ then the
following assertions are valid{\em :}
   \begin{enumerate}
   \item[(i)] $\slamj$ is a $2$-isometry
satisfying the kernel condition if and only if $\tcal$
is graph isomorphic either to $\zbb_+$ or to $\zbb,$
   \item[(ii)] if $\tcal$ is rooted, then $\slamj$ is
a Brownian isometry if and only if $\tcal$ is graph
isomorphic to $\zbb_+,$
   \item[(iii)] if $\tcal$ is rooted, then $\slamj$
is a quasi-Brownian isometry if and only if either
$\tcal$ is graph isomorphic to $\zbb_+$ or $\tcal$ is
a quasi-Brownian directed tree,
   \item[(iv)] if $\tcal$ is rootless, then $\slamj$
is a quasi-Brownian isometry if and only if $\slamj$
is a Brownian isometry, or equivalently, if and only
if either $\tcal$ is graph isomorphic to $\zbb$ or
$\tcal$ is a quasi-Brownian directed tree.
   \end{enumerate}
   \end{proposition}
   \begin{proof}
Since directed trees admitting $2$-isometric weighted
shifts are automatically leafless, we may assume
without loss of generality that $\tcal$ is leafless.

(i) It suffices to prove the ``only if'' part. Assume
that $\slamj\in \B(\ell^2(V))$ is a $2$-isometry
satisfying the kernel condition. First, observe that
$\slamj$ satisfies the condition \eqref{hypo+} for
some $\{\alpha_v\}_{v\in V} \subseteq \rbb_+$ (see
Remark~ \ref{codla2}). In view of Proposition~
\ref{bilat} and Lemma~ \ref{Hay2}, we can assume that
$\tcal$ has a root. It follows from Lemma~
\ref{basicws}(i) and the implication
(i)$\Rightarrow$(ii) of Proposition~ \ref{2iso-kc}
that $\deg{v} = \xi_n(\sqrt{\deg{\rot}}\,)^2$ for all
$v \in \childn{n}{\rot}$ and $n\in \zbb_+$. Hence, by
\eqref{ind1}, we see that $\deg{v}=\deg{w}$ for all
$v,w \in \child{u}$ and $u\in V$. This combined with
Lemmata~ \ref{basicws}(ii) and \ref{adja1}(i) (the
latter applied to $\beta_u\equiv 1$) implies that
$\sup_{u\in V} \deg u=\|\slamj\|^2$ and
   \begin{align*}
\deg{u} \deg{w} = \sum_{v \in \child{u}} \deg{v} = 2
\deg{u} - 1, \quad u \in V, \, w\in \child{u}.
   \end{align*}
As a consequence, we deduce that $\deg{u}=1$ for all
$u\in V$, which by Lemma~ \ref{Hay2} implies that
$\tcal$ is graph isomorphic to $\zbb_+.$

Before proving the assertions (ii)-(iv), we show that
   \begin{align} \label{sobota}
   \begin{minipage}{65ex}
if $\slamj\in \B(\ell^2(V)),$ then $\triangle_{\slamj}
\Ge 0;$ \\ if additionally $\triangle_{\slamj}
\slamj=\triangle_{\slamj}^{1/2} \slamj
\triangle_{\slamj}^{1/2}$, then \eqref{Hak2} holds.
   \end{minipage}
   \end{align}
Indeed, by Lemma~ \ref{basicws}, we see that
$\triangle_{\slamj} \Ge 0$ and
   \begin{align*}
\triangle_{\slamj} \slamj(e_u) = \triangle_{\slamj}
\bigg(\sum_{v\in \child{u}} e_v\bigg) = \sum_{v\in
\child{u}} (\deg{v}-1)e_v, \quad u \in V.
   \end{align*}
Similarly
   \begin{align*}
\triangle_{\slamj}^{1/2} \slamj
\triangle_{\slamj}^{1/2} (e_u) & = (\deg{u}-1)^{1/2}
\triangle_{\slamj}^{1/2} \slamj (e_u)
   \\
& = (\deg{u}-1)^{1/2} \sum_{v\in \child{u}}
(\deg{v}-1)^{1/2} e_v, \quad u \in V.
   \end{align*}
Hence, $\triangle_{\slamj}
\slamj=\triangle_{\slamj}^{1/2} \slamj
\triangle_{\slamj}^{1/2}$ if and only if
   \begin{align} \label{HaHa}
(\deg{u}-1)^{1/2} (\deg{v}-1)^{1/2} = (\deg{v}-1),
\quad v \in \child{u}, \, u\in V.
   \end{align}
This implies \eqref{sobota}.

(ii) This is a direct consequence of Proposition~
\ref{briai} and Lemma~ \ref{Hay2}.

(iii) Suppose $\tcal$ is rooted. Assume $\slamj$ is a
quasi-Brownian isometry. Set $l=\deg{\omega}.$
Clearly, $l<\infty.$ If $l=1,$ then in view of
\eqref{ind1}, \eqref{HaHa} and Lemma~ \ref{Hay2},
$\tcal$ is graph isomorphic to $\zbb_+.$ If $l \Ge 2,$
then by \eqref{sobota} and Lemmata~ \ref{qbchar} and
\ref{adja1}(i), $\tcal$ is a quasi-Brownian directed
tree. The converse implication is obvious in the case
when $\tcal$ is graph isomorphic to $\zbb_+$ (see
Lemma~ \ref{Hay2}). In turn, if $\tcal$ is a rooted
quasi-Brownian directed tree, then \eqref{HaHa} holds,
and consequently $\triangle_{\slamj}
\slamj=\triangle_{\slamj}^{1/2} \slamj
\triangle_{\slamj}^{1/2}$, which in view of Lemmata~
\ref{qbchar} and \ref{adja1}(i) shows that $\slamj$ is
a quasi-Brownian isometry.

(iv) Suppose $\tcal$ is rootless. Assume $\slamj$ is a
quasi-Brownian isometry. By Lemma~ \ref{Hay2}, we may
assume that there exists $u_0\in V$ such that
$l:=\deg{u_0} \in \{2,3,4,\ldots\}.$ Set $u_n=
\mathsf{par}^n(u_0)$ for $n\in \nbb.$ It follows from
\eqref{Hak2} (see \eqref{sobota}) that
$\deg{u_n}=\deg{u_{n+1}}$ for all $n\in \zbb_+.$ This
implies that $\deg{u_n}=l$ for all $n\in \zbb_+.$
Applying \eqref{sobota} and Lemmata~ \ref{qbchar} and
\ref{adja1}(i), we deduce that for every $n\in
\zbb_+,$ the rooted directed tree
   \begin{align*}
\tcal_{\des{u_n}}:=\big(\des{u_n}, (\des{u_n} \times
\des{u_n}) \cap E\big)
   \end{align*}
is a quasi-Brownian directed tree of valency $l.$ This
together with \cite[Proposition~ 2.1.6(iii)]{JJS}
implies that the directed tree $\tcal$ itself is a
quasi-Brownian directed tree of valency $l.$

In view of Corollary~ \ref{brareqbr}, it remains to
show that the adjacency operator of a rootless
quasi-Brownian directed tree is a Brownian isometry.
Clearly \eqref{Hak3} holds, so by Lemma~
\ref{adja1}(i), $\slamj$ is a $2$-isometry. It follows
from Lemma~ \ref{basicws} and the definition of a
quasi-Brownian directed tree that
   \allowdisplaybreaks
   \begin{align*}
\triangle_{\slamj} \triangle_{\slamj^*}
\triangle_{\slamj} (e_u) & =
(\deg{u}-1)\triangle_{\slamj} \triangle_{\slamj^*}
(e_u)
   \\
& = (\deg{u}-1)\triangle_{\slamj} \bigg(\sum_{v\in
\child{\parent{u}}\setminus \{u\}}e_v \bigg)
   \\
& = \sum_{v\in \child{\parent{u}} \setminus \{u\}}
(\deg{u}-1)(\deg{v}-1) e_v = 0, \quad u\in V.
   \end{align*}
Hence $\slamj$ is a Brownian isometry. This completes
the proof.
   \end{proof}
   Combining Lemmata~ \ref{qbchar} and \ref{adja1}(i)
with Proposition~ \ref{stwfa}, we get the following.
   \begin{corollary} \label{newchar}
Let $\slamj\in \B(\ell^2(V))$ be the adjacency
operator of a rooted directed tree $\tcal$ such that
$\deg{\omega} \in \{2,3,4, \ldots\}$. Set
$l=\deg{\omega}$. If $\slamj$ is a $2$-isometry, then
the following conditions are equivalent{\em :}
   \begin{enumerate}
   \item[(i)] $\slamj$ is a quasi-Brownian isometry,
   \item[(ii)] $\tcal$ is a quasi-Brownian  directed
tree,
   \item[(iii)] each vertex $u\in V^{\circ}$ has degree $1$ or
$l.$
   \end{enumerate}
   \end{corollary}
Below we concentrate on $2$-isometric adjacency
operators of directed trees which satisfy certain
degree constraints (see Figure \ref{figure-p1} for an
illustration of parts (i) and (ii) of Lemma~
\ref{constant}).
   \begin{figure}[t]
   \begin{center}
   \subfigure {
\includegraphics[scale=0.30]{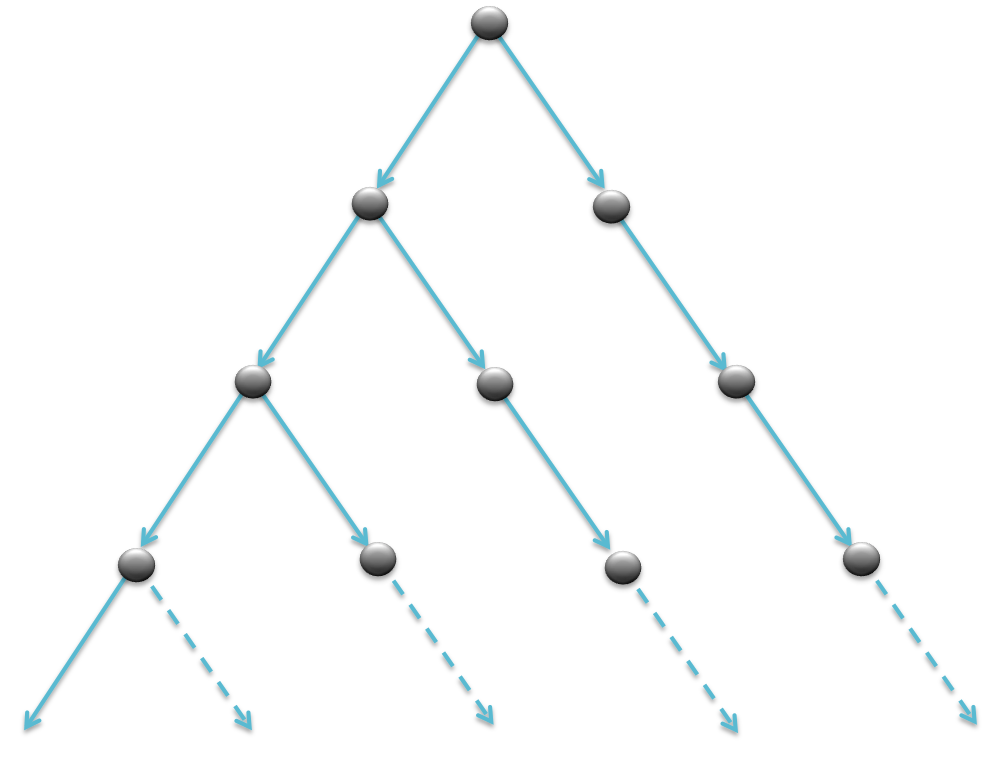}
} \subfigure {
\includegraphics[scale=0.25]{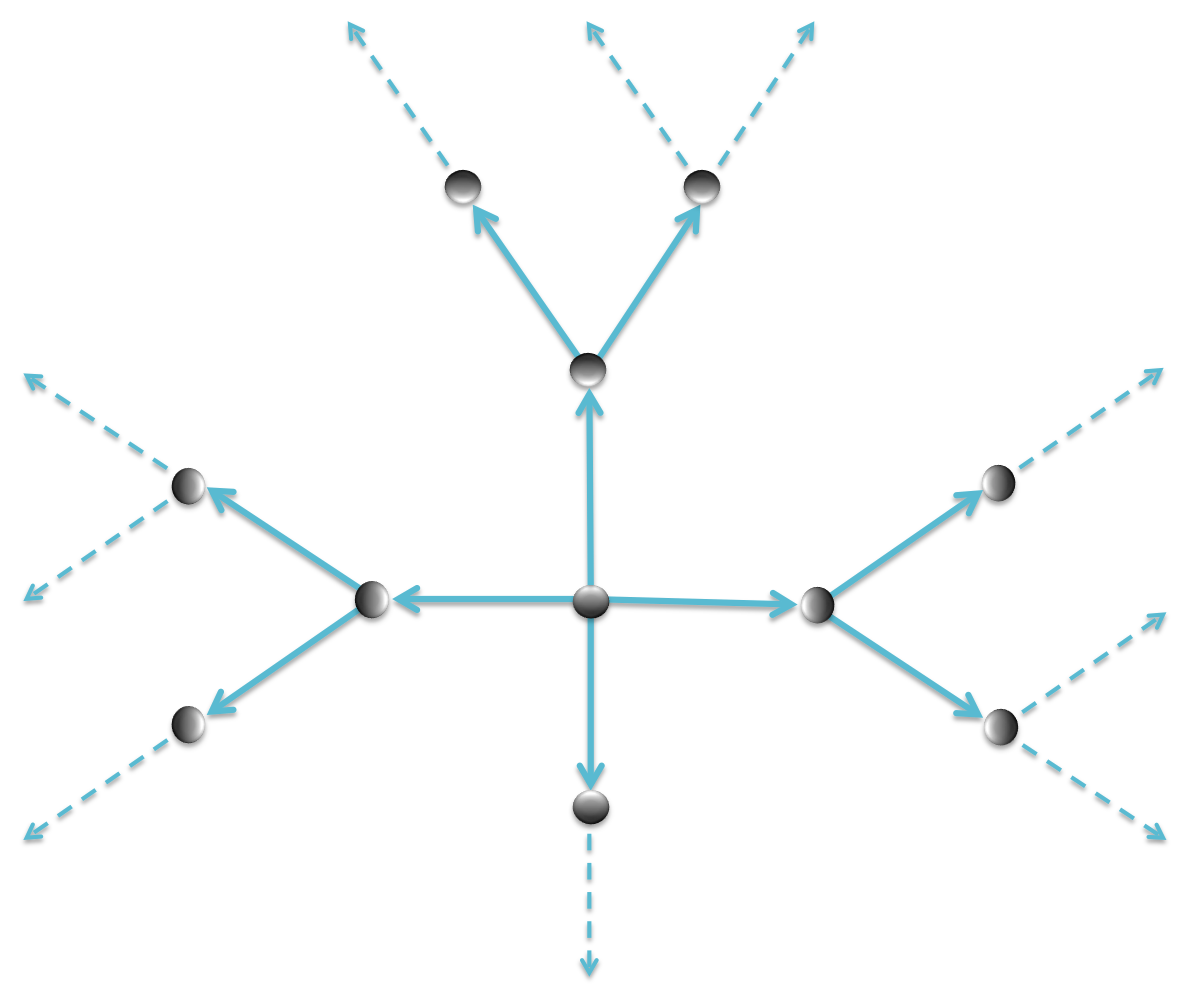}
} \caption{Examples of directed trees satisfying the
conditions (i) and (ii) of Lemma~ \ref{constant} with
$l=2$ and $l=4,$ respectively.} \label{figure-p1}
   \end{center}
   \end{figure}
   \begin{lemma} \label{constant}
Let $\mathscr T$ be a rooted, leafless and locally
finite directed tree. Set $l=\deg{\rot}$. Suppose
$\slamj\in \B(\ell^2(V))$ is a $2$-isometry and
$d_{\slamj^{\prime}}$ is given by {\em \eqref{defdun}}
with $\slamj^{\prime}$ in place of $\slam$ $($see
Lemma~ {\em \ref{cisws}}$)$. Then the following
assertions~ hold.
   \begin{enumerate}
   \item[(i)]  Assume that each vertex $u\in V^{\circ}$ has
degree $1$ or $l$. Then
   \begin{align} \label{moment1}
d_{\slamj^{\prime}}(u,n)=
   \begin{cases}
1 & \text{if } \deg{u}=1,
   \\[.5ex]
\frac{1}{l+1}\left(1 + l^{1-2n}\right) & \text{if }
\deg{u}=l,
   \end{cases}
\quad n \in \zbb_+, \, u \in V.
   \end{align}
   \item[(ii)] Assume that $l\Ge 2$ and $l- 1$ vertices of
$\child{\rot}$ have degree $2$. Then each $u\in V$ has
degree $1$, $2$ or $l$ and
   \begin{align}  \label{moment-1}
d _{\slamj^{\prime}}(u,n) =
   \begin{cases}
1 & \text{if } \deg{u}=1, \, n \in \zbb_+,
   \\[.5ex]
\frac{1}{3}\left(1 + 2^{1-2n}\right) & \text{if }
\deg{u}=2, \, n \in \zbb_+,
   \\[.5ex]
\frac{1}{3l^2}\left(l + 2 + 2(l-1) 2^{2(1-n)}\right) &
\text{if } \deg{u}=l, \, n\in \nbb,
   \end{cases}
\quad u \in V.
   \end{align}
   \end{enumerate}
   \end{lemma}
   \begin{proof}
   Since $\slamj$ is a $2$-isometry satisfying
\eqref{beta} with $\beta_u\equiv 1$, Lemma~
\ref{adja1} yields
   \allowdisplaybreaks
   \begin{gather} \label{2isoadj}
\sum_{v \in \child{u}} \deg{v} = 2 \deg{u} - 1, \quad
u \in V,
   \\  \notag
d _{\slamj^{\prime}}(u,0)=1, \quad u\in V,
   \\ \label{recur2}
d _{\slamj^{\prime}}(u, n+1) = \frac{1}{(\deg{u})^{2}}
\sum_{v \in \child{u}} d _{\slamj^{\prime}}(v, n),
\quad u \in V, \, n \in \zbb_+.
   \end{gather}
These three formulas and appropriate degree
assumptions are all that are needed to prove (i) and
(ii). It follows from \eqref{2isoadj} that
   \begin{align*}
\text{$\deg{v}=1$ whenever $v\in \child{u}$ and
$\deg{u}=1.$}
   \end{align*}
Hence, by a simple induction argument based on
\eqref{recur2}, we get
   \begin{align}  \label{jeden}
\text{$d _{\slamj^{\prime}}(u,n)=1$ for all
$n\in\zbb_+$ whenever $\deg{u}=1$.}
   \end{align}
This proves the degree $1$ case of \eqref{moment1} and
\eqref{moment-1}.

(i) Without loss of generality, we may assume that
$l\Ge 2.$ Suppose that $\deg{u}=l$. Then
\eqref{2isoadj} and the assumption of (i) yield
   \begin{align}\label{1-l-l-1}
\text{$\child{u}$ consists of $l-1$ vertices of degree
$1$ and one vertex of degree $l$.}
   \end{align}
We will verify the bottom formula in \eqref{moment1}
by using induction on $n$. The case of $n=0$ is
obvious. Assume that it holds for a fixed $n\in
\zbb_+$. Then
   \begin{align*}
d _{\slamj^{\prime}}(u,n+1) &
\overset{\eqref{recur2}}= l^{-2} \sum_{v \in
\child{u}} d _{\slamj^{\prime}}(v, n)
   \\
&\hspace{.4ex}\overset{(*)}= l^{-2} \left((l -1) +
\frac{1}{l+1}\left(1 + l^{1-2n}\right)\right)
   \\
&\hspace{.7ex}= \frac{1}{l+1}\left(1 +
l^{1-2(n+1)}\right),
   \end{align*}
where ($*$) follows from \eqref{1-l-l-1},
\eqref{jeden} and the induction hypothesis. This
completes the proof of the assertion (i).

(ii) An induction argument based on \eqref{ind1} and
\eqref{2isoadj} shows that\footnote{\;The statement
(a) holds without the degree assumption of (ii).}
   \begin{enumerate}
   \item[(a)] if $u\in V$ is of degree $2$, then
$\child{u}$ consists of one vertex of degree $1$ and
one vertex of degree $2$,
   \item[(b)] $\child{\rot}$ consists of one vertex
of degree $1$ and $l- 1$ vertices of degree $2$,
   \item[(c)] if $u\in
V^{\circ}$, then $\deg{u}\in \{1,2\}$,
   \item[(d)] if $l\Ge 3$, then $\rot$ is the
only vertex of $V$ of degree $l$.
   \end{enumerate}
If $\deg{u}=2$, then by (c) the directed tree
$\big(\des{u}, (\des{u} \times \des{u}) \cap E\big)$
satisfies the assumption of (i) for $l=2$, and thus
the degree $2$ case of \eqref{moment-1} follows from
(i) by applying Proposition~ \ref{restrcd} (cf.\ the
proof of Theorem~ \ref{main2}). By (c) and (d), it
remains to consider the case of $\deg{u}=l\Ge 3$,
i.e., $u=\rot$. If $n\in \nbb$, then
   \begin{align*}
d _{\slamj^{\prime}}(\rot,n)
&\overset{\eqref{recur2}}= l^{-2} \sum_{v \in
\child{\rot}} d _{\slamj^{\prime}}(v, n-1)
   \\
& \hspace{.4ex}\overset{(*)}= l^{-2} \left(1+
\frac{l-1}{3}\left(1 + 2^{1-2(n-1)}\right)\right)
   \\
& \hspace{.8ex} = \frac{1}{3l^2}\left(l + 2 + 2(l-1)
2^{2(1-n)}\right),
   \end{align*}
where ($*$) follows from (b), \eqref{jeden} and the
degree $2$ case of \eqref{moment-1}. This completes
the proof of Lemma~ \ref{constant}.
   \end{proof}
   \begin{remark}
An inspection of the proof of Lemma~ \ref{constant}
reveals that
   \begin{enumerate}
   \item[$\bullet$] under the assumption of (i),
there are infinitely many vertices of degree $1$ and,
if $l \Ge 2$, there are infinitely many vertices of
degree $l$ (cf.\ Corollary~ \ref{newchar}),
   \item[$\bullet$] under the assumption of (ii),
there are infinitely many vertices of degree $1$ and
infinitely many vertices of degree $2$; if $l\Ge 3$,
then $\rot$ is the only vertex of degree $l$.
   \hfill $\diamondsuit$
   \end{enumerate}
   \end{remark}
Let us make some comments regarding the condition
\eqref{2isoadj}.
   \begin{remark}
Assume that $\tcal$ is a rooted, leafless and locally
finite directed tree. If $\deg{u}=1$ or $\deg{u}=2$,
then $d _{\slamj^{\prime}}(u,n)$ can be calculated
explicitly. If $\deg{u}=3$, then, by \eqref{2isoadj},
$\child{u}$ may consists of one vertex of degree $3$
and two vertices of degree $1$, or two vertices of
degree $2$ and one vertex of degree $1$. Now we have
two possibilities either the degree $3$ case appear
infinitely many times, or the process stops after a
finite number of steps and then the degree $1$ and the
degree $2$ cases appear both infinitely many times. If
$\deg{u}=4$, then, by \eqref{2isoadj} again,
$\child{u}$ may consists of one vertex of degree $4$
and three vertices of degree $1$, or one vertex of
degree $3$, one vertex of degree $2$ and two vertices
of degree $1$, or three vertices of degree $2$ and one
vertex of degree $1$, and so on.
   \hfill $\diamondsuit$
   \end{remark}
Now we show that the Cauchy dual subnormality problem
has an affirmative solution for certain adjacency
operators which, in general, do not satisfy the kernel
condition (see Theorem~ \ref{constant-t}). This leads
to Hausdorff moment sequences which are structurally
different from those in Section~ \ref{Sec10}. We also
answer the question of when such adjacency operators
are Brownian or quasi-Brownian isometries.
   \begin{theorem} \label{constant-t}
Let $\mathscr T$ be a rooted, leafless and locally
finite directed tree. Set $l=\deg{\rot}$. Suppose that
$\slamj\in \B(\ell^2(V))$ is a $2$-isometry and one of
the following two conditions holds{\em :}
   \begin{enumerate}
   \item[(i)] each vertex $u\in V^{\circ}$ has degree $1$ or
$l$,
   \item[(ii)] $l\Ge 2$ and $l- 1$ vertices of
$\child{\rot}$ have degree $2$.
   \end{enumerate}
Then the following statements are valid{\em :}
   \begin{enumerate}
   \item[\ding{202}] the Cauchy dual $\slamj^{\prime}$
of $\slamj$ is a subnormal contraction,
   \item[\ding{203}] $\slamj$ satisfies the kernel
condition if and only if $l=1,$
   \item[\ding{204}] if {\em (i)} holds, then $\slamj$
is always a quasi-Brownian isometry; moreover $\slamj$
is a Brownian isometry if and only if $l=1,$
   \item[\ding{205}] if {\em (ii)}
holds, then $\slamj$ is never a Brownian isometry;
moreover, $\slamj$ is a quasi-Brownian isometry if and
only if $l=2.$
   \end{enumerate}
   \end{theorem}
   \begin{proof}
\ding{202} It follows from Lemma~ \ref{cisws}, the
formula \eqref{defdun} (applied to $\slamj^{\prime}$)
and \cite[Theorem~ 6.1.3]{JJS} that $\slamj^{\prime}$
is subnormal if and only if
   \begin{align} \label{stims}
\text{$\{d _{\slamj^{\prime}}(u,n)\}_{n=0}^{\infty}$
is a Stieltjes moment sequence for every $u\in V.$}
   \end{align}
If (i) holds, then \eqref{stims} follows easily from
\eqref{moment1}.

Now, assume that (ii) holds. In view of
\eqref{moment-1},
$\{d_{\slamj^{\prime}}(u,n)\}_{n=0}^{\infty}$ is a
Stieltjes moment sequence whenever $u$ is of degree
$1$ or $2$. The only nontrivial case is when $l\Ge 3$
and $\deg{u}=l$, i.e., $u=\rot$ (see the statement (d)
of the proof of Lemma~ \ref{constant}). Then, by
\eqref{moment-1}, we have
   \begin{align} \label{slon}
d _{\slamj^{\prime}}(\rot,n+1) = \frac{2(l-1)}{3l^2}
\frac{1}{4^{n}} + \frac{l + 2}{3l^2}, \quad n\in
\zbb_+.
   \end{align}
This implies that $\{d
_{\slamj^{\prime}}(\rot,n+1)\}_{n=0}^{\infty}$ is a
Stieltjes moment sequence with the representing
measure $\mu= \frac{2(l-1)}{3l^2} \delta_{\frac{1}{4}}
+ \frac{l + 2}{3l^2} \delta_1$. Since $l\Ge 3$, it is
easily seen that
   \begin{align*}
\int_{\rbb_+} \frac{1}{t} d \mu(t) = \frac{3l-2}{l^2}
\Le 1.
   \end{align*}
Hence, by Lemma~ \ref{bext}, $\{d_{\slamj^{\prime}}
(\rot,n)\}_{n=0}^{\infty}$ is a Stieltjes moment
sequence.

\ding{203} This is an immediate consequence of
Proposition~ \ref{stwfa} and Lemma~ \ref{Hay2}.

\ding{204} This is a consequence of Lemma~ \ref{Hay2},
Corollary~ \ref{newchar} and Proposition~ \ref{stwfa}
(cf.\ Proposition~ \ref{briai}).

\ding{205} Assume that (ii) holds. It follows from
Proposition~ \ref{stwfa}(ii) that $\slamj$ is never a
Brownian isometry. Suppose now $\slamj$ is a
quasi-Brownian isometry. Then, by Theorem~
\ref{BrownianG}, we have
   \begin{align} \label{Sam3Zen}
\slamj'^{*n}\slamj'^n =
(I+\slamj^*\slamj)^{-1}(I+(\slamj^*\slamj)^{1- 2n}),
\quad n \in \zbb_+.
   \end{align}
It follows from Lemma~ \ref{basicws}(v) that
$\slamj^*\slamj$ is a diagonal operator (with respect
to the orthogonal basis $\{e_u\}_{u\in V}$) with
diagonal elements $\{\deg u\}_{u\in V}.$ Hence,
\eqref{Sam3Zen} yields
   \begin{align} \label{krowa1}
\|\slamj'^n e_{\omega}\|^2 = \frac{1+l^{1-2n}}{1+l},
\quad n\in \nbb.
   \end{align}
According to Lemma~ \ref{constant}(ii) and
\eqref{defdun} (with $\slamj'$ in place of $\slam$),
we have
   \begin{align} \label{krowa2}
\|\slamj'^n e_{\omega}\|^2 = \frac{l + 2 + 2(l-1)
2^{2(1-n)}}{3l^2}, \quad n\in \nbb.
   \end{align}
Comparing \eqref{krowa1} with \eqref{krowa2} and
passing to the limit as $n\to\infty,$ we see that
$l=2$ (recall that $l\Ge 2$). Conversely, if $l=2,$
then by Lemma~ \ref{constant}(ii), each $u\in V$ is of
degree $1$ or $2.$ Hence, in view of Corollary~
\ref{newchar}, $\slamj$ is a quasi-Brownian isometry.
This completes the proof.
   \end{proof}
Below we construct a rooted and leafless directed tree
$\tcal$ such that
   \begin{align} \label{ajaja}
   \left.
   \begin{minipage}{70ex}
   \begin{enumerate}
   \item[$\bullet$] $\slamj\in \B(\ell^2(V))$ is a $2$-isometry,
   \item[$\bullet$] $\slamj'$ is a subnormal contraction,
   \item[$\bullet$] $\slamj$ does not satisfy the
kernel condition,
   \item[$\bullet$] $\slamj$ is not a
quasi-Brownian isometry.
   \end{enumerate}
   \end{minipage}
   \right\}
   \end{align}
In fact, we show that for every integer $l\Ge 2,$
there exists a rooted and leafless directed tree with
$2$-isometric adjacency operator, which satisfies the
condition (ii) of Theorem~ \ref{constant-t}. Similar
construction can be performed in the case of the
condition (i) of Theorem~ \ref{constant-t}; the
resulting directed tree is either $\zbb_+$ if $l=1$ or
a quasi-Brownian directed tree of valency $l$ if $l\Ge
2$.
   \begin{example} \label{nbnkcsub}
Let us fix $l \in \{2,3,4, \ldots\}$. Using
\eqref{ind1}, one can construct inductively a rooted
and leafless directed tree $\tcal$ with the following
properties (see the right subfigure of Figure
\ref{figure-p1} in which $l=4$):
   \begin{enumerate}
   \item[(i)] $\deg {\rot}=l$,
   \item[(ii)] $\child{\rot}$ consists of one
vertex of degree $1$ and $l-1$ vertices of degree $2$,
   \item[(iii)]
for each vertex $u\in V^{\circ}$ of degree $1,$
$\child{u}$ consists of one vertex of degree~ $1,$
   \item[(iv)]
for each vertex $u\in V^{\circ}$ of degree $2$,
$\child{u}$ consists of one vertex of degree $1$ and
one of degree $2.$
   \end{enumerate}
Clearly, the directed tree $\tcal$ satisfies the
condition (ii) of Theorem~ \ref{constant-t}. Let
$\slamj$ be the adjacency operator of $\tcal.$ Since
$\|\slamj e_u\|^2 \Le l$ for all $u\in V$, we infer
from Lemma~ \ref{basicws} that $\slamj \in
\B(\ell^2(V))$ and $\|\slamj\|^2 = l$. It follows from
\mbox{(i)-(iv)} that
   \begin{align*}
\sum_{v\in \child{u}} \deg{v} = 2 \deg u -1, \quad u
\in V.
   \end{align*}
This combined with Lemma~ \ref{adja1}(i) implies that
$\slamj$ is a $2$-isometry. If $l\Ge 3,$ then by
Theorem~ \ref{constant-t}, the adjacency operator
$\slamj$ satisfies \eqref{ajaja}. In turn, if $l=2,$
then by Theorem~ \ref{constant-t} again, $\slamj$ is a
quasi-Brownian isometry that satisfies all but the
last condition in \eqref{ajaja}.
   \hfill $\diamondsuit$
   \end{example}
   We conclude this section by showing that the Cauchy
dual subnormality problem has a negative solution in
the class of adjacency operators.
   \begin{figure}[t]
   \begin{center}
\subfigure {
\includegraphics[scale=0.58]{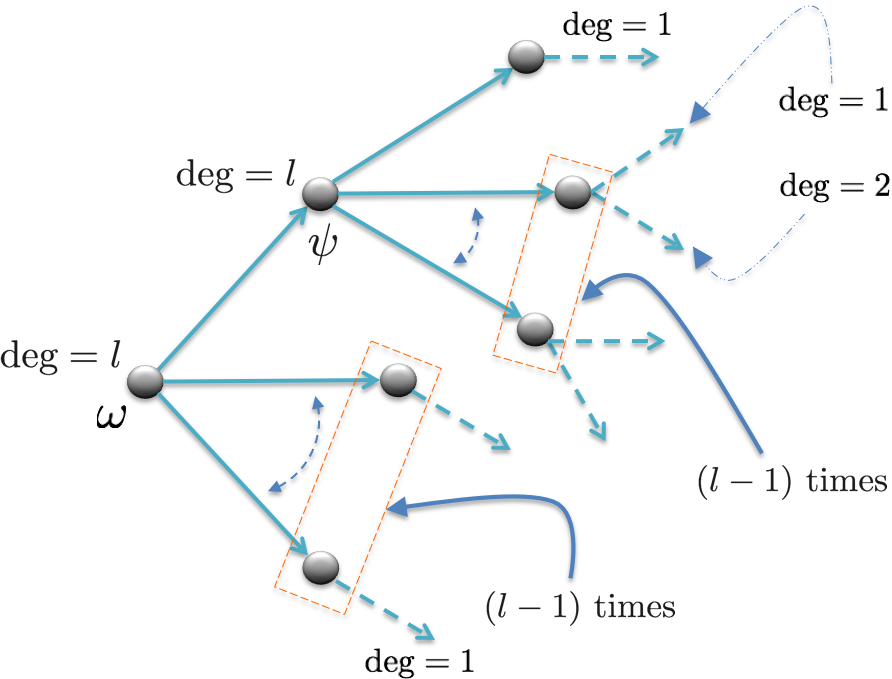}
} \caption{The directed tree appearing in Example~
\ref{przadj}.} \label{Fig1e-v2}
   \end{center}
   \end{figure}
   \begin{example} \label{przadj}
We begin by constructing an appropriate directed tree.
Let us fix $l \in \{3,4,5, \ldots\}$. Using
\eqref{ind1}, one can construct inductively a rooted
and leafless directed tree $\tcal$ with the following
properties (see Figure \ref{Fig1e-v2}):
   \begin{enumerate}
   \item[(i)] $\deg {\rot}=l$,
   \item[(ii)] $\child{\rot}$ consists of $l-1$ vertices of
degree $1$ and one vertex, say $\psi$, of degree~ $l$,
   \item[(iii)] $\child{\psi}$ consists of one
vertex of degree $1$ and $l-1$ vertices of degree $2$,
   \item[(iv)]
for each vertex $u\in V$ of degree $1$, $\child{u}$
consists of one vertex of degree~ $1$,
   \item[(v)]
for each vertex $u\in V$ of degree $2$, $\child{u}$
consists of one vertex of degree $1$ and one of degree
$2$.
   \end{enumerate}
Let $\slamj$ be the adjacency operator of the directed
tree $\tcal$. Exactly as in Example~ \ref{nbnkcsub},
we verify that $\slamj \in \B(\ell^2(V))$ and
$\|\slamj\|^2 = l$. It is also a routine matter to
verify that the properties \mbox{(i)-(v)} yield
   \begin{align*}
\sum_{v\in \child{u}} \deg{v} = 2 \deg u -1, \quad u
\in V.
   \end{align*}
This together with Lemma~ \ref{adja1}(i) implies that
$\slamj$ is a $2$-isometry. Set
$\tcal_{\psi}=\big(\des{\psi}, (\des{\psi} \times
\des{\psi}) \cap E\big)$. Then $\tcal_{\psi}$ is a
rooted and leafless directed tree with the root
$\psi$. Recall that by Lemma~ \ref{basicws}, the space
$\ell^2(\des{\psi})$ is invariant for $\slamj$ and
$\slamj^*\slamj$, and that the restriction of $\slamj$
to $\ell^2(\des{\psi})$ coincides with the adjacency
operator $S_{\psi,\mathbb 1}$ of the directed tree
$\tcal_{\psi}$ (cf.\ the proof of Theorem~
\ref{main2}). Hence $S_{\psi,\mathbb 1}\in
\B(\ell^2(\des{\psi}))$ is a $2$-isometry, and by
Proposition~ \ref{restrcd}, $S_{\psi,\mathbb
1}^{\prime}$ coincides with the restriction of
$\slamj^{\prime}$ to $\ell^2(\des{\psi})$. As a
consequence, we see that $d_{S_{\psi,\mathbb
1}^{\prime}}(u,n)=d_{\slamj^{\prime}}(u,n)$ for all $u
\in \des{\psi}$ and $n\in \zbb_+$. Applying Theorem~
\ref{constant-t} (the case (ii)) to the directed tree
$\tcal_{\psi}$ and its adjacency operator
$S_{\psi,\mathbb 1}$ and using \cite[Theorem~
6.1.3]{JJS}, we deduce that
$\{d_{\slamj^{\prime}}(\psi,n)\}_{n=0}^{\infty}$ is a
Stieltjes moment sequence. A careful look at the proof
of Theorem~ \ref{constant-t} (use \eqref{slon} with
$\psi$ and $S_{\psi,\mathbb 1}^{\prime}$ in place of
$\rot$ and $\slamj^{\prime}$, respectively) shows that
the positive Borel measure $\mu$ on $\rbb_+$ defined
by
   \begin{align} \label{slon2}
\mu = \frac{2(l-1)}{3l^2} \delta_{\frac{1}{4}} +
\frac{l+2}{3l^2} \delta_1
   \end{align}
is a unique representing measure of
$\{d_{\slamj^{\prime}}(\psi,n+1)\}_{n=0}^{\infty}$. It
follows from Lemma~ \ref{bext} that the positive Borel
measure $\nu$ on $\rbb_+$ given by
   \begin{align*}
\nu(\varDelta) = \int_{\varDelta} \frac 1 t d \mu(t) +
\Big(1-\int_{\rbb_+} \frac 1 t d \mu(t)\Big)
\delta_0(\varDelta), \quad \varDelta \in
\borel{\rbb_+},
   \end{align*}
is a representing measure of the Stieltjes moment
sequence
$\{d_{\slamj^{\prime}}(\psi,n)\}_{n=0}^{\infty}$. This
and \eqref{slon2} imply that
   \begin{align} \label{slon3}
\nu = \frac{(l-1)(l-2)}{l^2} \delta_0 +
\frac{8(l-1)}{3l^2} \delta_{\frac{1}{4}} +
\frac{l+2}{3l^2} \delta_1.
   \end{align}
Using (iv) and Lemma~ \ref{adja1}(ii), we verify that
\eqref{jeden} holds. Hence, by applying Lemma~
\ref{adja1}(ii) to $u=\rot$, we get
   \allowdisplaybreaks
   \begin{align*}
d_{\slamj^{\prime}}(\rot, n+1) & = \frac{1}{l^{2}}
\sum_{v \in \child{\rot}} d_{\slamj^{\prime}}(v, n)
   \\
&\hspace{-.3ex}\overset{\mathrm{(ii)}}=
\frac{1}{l^{2}} \left(l-1 + d_{\slamj^{\prime}}(\psi,
n)\right)
   \\
& = \frac{1}{l^{2}} \left(l-1 + \int_{\rbb_+} t^n
d\nu(t)\right)
   \\
&= \int_{\rbb_+} t^n d \rho(t), \quad n \in \zbb_+,
   \end{align*}
where, by \eqref{slon3}, $\rho$ is the positive Borel
measure on $\rbb_+$ defined by
   \begin{align*}
\rho = \frac{(l-1)(l-2)}{l^4} \delta_0 +
\frac{8(l-1)}{3l^4} \delta_{\frac{1}{4}} +
\frac{3(l-1)l^2+l+2}{3l^4} \delta_1.
   \end{align*}
This means that
$\{d_{\slamj^{\prime}}(\rot,n+1)\}_{n=0}^{\infty}$ is
a Stieltjes (in fact a Hausdorff) moment sequence with
the unique representing measure $\rho$. Since
$\rho(\{0\}) > 0$ (because $l\Ge 3$), we infer from
Lemma~ \ref{bext} that
$\{d_{\slamj^{\prime}}(\rot,n)\}_{n=0}^{\infty}$ is
not a Stieltjes moment sequence. Hence, by
\cite[Theorem~ 6.1.3]{JJS}, the operator
$\slamj^{\prime}$ is not subnormal.
   \hfill $\diamondsuit$
   \end{example}
   \section*{Acknowledgement}
A part of this paper was written while the second
author visited Jagiellonian University in Summer of
2017. He wishes to thank the faculty and the
administration of this unit for their warm
hospitality.
   
   \end{document}